\font \sevenrm=cmr7
\font \fiverm=cmr5
\documentclass[12pt, english]{amsart}
\input epsf.sty
\usepackage{amsfonts}
\usepackage{amssymb}
\usepackage{amsmath}
\usepackage{amsthm}
\usepackage{amscd}
\usepackage[all]{xy} 
\usepackage{epsfig,exscale}
\usepackage{color}
\usepackage{graphicx}
\usepackage{xspace}
\usepackage{axodraw4j}
\usepackage{colordvi}

\newcommand{\nc}{\newcommand}

{\everymath{\displaystyle\everymath{}}\array}%
{\endarray}
\newtheorem{theorem}{Theorem}
\newtheorem{definition}{Definition}

\newtheorem{proposition}{Proposition}
\newtheorem{ex}{Example}
\newtheorem{remark}{Remark}
\nc{\comment}[1]{[[{\tt #1}]] }
\nc{\Cal}[1]{{\mathcal {#1}}}
\nc{\mop}[1]{\mathop{\hbox {\rm #1} }\nolimits}
\nc{\gmop}[1]{\mathop{\hbox {\bf #1} }\nolimits}

\nc{\smop}[1]{\mathop{\hbox {\sevenrm #1} }\nolimits}
\nc{\ssmop}[1]{\mathop{\hbox {\fiverm #1} }\nolimits}
\nc{\mopl}[1]{\mathop{\hbox {\rm #1} }\limits}
\nc{\smopl}[1]{\mathop{\hbox {\sevenrm #1} }\limits}
\nc{\ssmopl}[1]{\mathop{\hbox {\fiverm #1} }\limits}
\nc{\frakg}{{\frak g}}
\nc{\g}[1]{{\frak {#1}}}
\def \restr#1{\mathstrut_{\textstyle |}\raise-6pt\hbox{$\scriptstyle #1$}}
\def \srestr#1{\mathstrut_{\scriptstyle |}\hbox to
-1.5pt{}\raise-4pt\hbox{$\scriptscriptstyle #1$}}
\nc{\wt}{\widetilde} \nc{\wh}{\widehat}
\nc{\redtext}[1]{\textcolor{red}{#1}}
\nc{\bluetext}[1]{\textcolor{blue}{#1}}
\nc\fleche[1]{\mathop{\hbox to #1 mm{\rightarrowfill}}\limits}
\nc{\ignore}[1]{}
\def\semi{\mathrel{\times}\kern -.85pt\joinrel\mathrel{\raise
1.4pt\hbox{${\scriptscriptstyle |}$}}}
\nc\R{{\mathbb R}}
\nc\N{{\mathbb N}}
\nc\inver{^{-1}}
\nc\point{\hbox{\bf .}}
\nc\un{\hbox{\bf 1}}

\def\graphearetenor{\,{\scalebox{0.25}{
\begin{picture}(130,6) (175,-221)
\SetWidth{3.0}
\SetColor{Black}
\Line(176,-218)(304,-218)
\end{picture}
}}\,}

\def\graphecroi{\,{\scalebox{0.20}{
\begin{picture}(672,54) (05,-131)
    \SetWidth{3.0}
    \SetColor{Black}
    \Line(20,-48)(686,-160)
    \Line(16,-160)(673,-48)
  \end{picture}  }}\,}
\def\graphecroipetit{\,{\scalebox{0.20}{
\begin{picture}(482,60) (95,-183)
    \SetWidth{3.0}
    \SetColor{Black}
    \Line(96,-116)(576,-178)
    \Line(96,-192)(574,-114)
\end{picture} }}\,}

\def\graphearetenorcroix{\,{\scalebox{0.25}{
\begin{picture}(130,34) (95,-143)
\SetWidth{3.0}
\SetColor{Black}
\Line(96,-126)(224,-126)
\Line(144.002,-110.002)(175.998,-141.998)\Line(175.998,-110.002)(144.002,-141.998)
\end{picture}
}}\,}
\def\graphearetecurcroix{\,{\scalebox{0.25}{
\begin{picture}(129,22) (129,-200)
\SetWidth{3.0}
\SetColor{Black}
\Photon(130,-190)(257,-187){7.5}{6}
\Line(181.002,-198.998)(200.998,-179.002)\Line(181.002,-179.002)(200.998,-198.998)
\end{picture}
}}\,}
\def\graphearetecurcroixun{\,{\scalebox{0.25}{
\begin{picture}(130,90) (111,-177)
\SetWidth{3.0}
\SetColor{Black}
\Photon(112,-164)(240,-164){7.5}{6}
\Line(191.998,-179.998)(160.002,-148.002)\Line(160.002,-179.998)(191.998,-148.002)
\Text(176,-212)[lb]{\Huge{\Black{$1$}}}
\end{picture}
}}\,}

\def\graphearetenorcroixzero{\,{\scalebox{0.25}{
\begin{picture}(130,34) (95,-143)
\SetWidth{3.0}
\SetColor{Black}
\Line(96,-126)(224,-126)
\Line(144.002,-110.002)(175.998,-141.998)\Line(175.998,-110.002)(144.002,-141.998)
\Text(156,-172)[lb]{\Huge{\Black{$0$}}}
\end{picture}
}}\,}
\def\graphearetenorcroixun{\,{\scalebox{0.25}{
\begin{picture}(130,34) (95,-143)
\SetWidth{3.0}
\SetColor{Black}
\Line(96,-126)(224,-126)
\Line(144.002,-110.002)(175.998,-141.998)\Line(175.998,-110.002)(144.002,-141.998)
\Text(156,-172)[lb]{\Huge{\Black{$1$}}}
\end{picture}
}}\,}
\def\graphearetecur{\,{\scalebox{0.25}{
\begin{picture}(133,20) (104,-209)
\SetWidth{3.0}
\SetColor{Black}
\Photon(105,-201)(236,-197){7.5}{7}
\end{picture}
}}\,}
\def\graphev{\,{\scalebox{0.25}{
\begin{picture}(66,34) (191,-191)
    \SetWidth{3.0}
    \SetColor{Black}
    \Line(192,-174)(224,-174)
    \Line(224,-174)(256,-158)
    \Line(224,-174)(256,-190)
\end{picture}}}\,}
\def\graphevcur{\,{\scalebox{0.15}{
\begin{picture}(160,166) (97,-122)
\SetWidth{3.0}
\SetColor{Black}
\Photon(98,-99)(195,-97){7.5}{5}
\Line(195,-99)(256,-37)
\Line(196,-100)(255,-161)
\end{picture}
}}\,}
\def\grapheexmdel{\,{\scalebox{0.25}{
\begin{picture}(162,52) (175,-209)
\SetWidth{3.0}
\SetColor{Black}
\Arc(256,-222)(16,270,630)
\Arc(256,-192)(34,-61.928,241.928)
\Line(288,-190)(336,-190)
\Line(224,-190)(176,-190)
\end{picture}
}}\,}
\def\graphecerc{\,{\scalebox{0.25}{
\begin{picture}(130,34) (191,-287)
\SetWidth{3.0}
\SetColor{Black}
\Arc(256,-270)(16,270,630)
\Line(272,-270)(320,-270)
\Line(240,-270)(192,-270)
\end{picture}
}}\,}
\def\graphecerccroizero{\,{\scalebox{0.25}{
\begin{picture}(162,63) (175,-101)
\SetWidth{3.0}
\SetColor{Black}
\Arc(256,-75)(35.777,243,603)
\Line(288,-75)(336,-75)
\Line(224,-75)(176,-75)
\Line(240.002,-122.998)(271.998,-91.002)\Line(240.002,-91.002)(271.998,-122.998)
\Text(256,-139)[lb]{\Huge{\Black{$0$}}}
\end{picture}
}}\,}
\def\graphecerccroiun{\,{\scalebox{0.25}{
\begin{picture}(162,63) (175,-101)
\SetWidth{3.0}
\SetColor{Black}
\Arc(256,-75)(35.777,243,603)
\Line(288,-75)(336,-75)
\Line(224,-75)(176,-75)
\Line(240.002,-122.998)(271.998,-91.002)\Line(240.002,-91.002)(271.998,-122.998)
\Text(256,-139)[lb]{\Huge{\Black{$1$}}}
\end{picture}
}}\,}
\def\graphedeltasansindice{\,{\scalebox{0.20}{
\begin{picture}(226,68) (111,-161)
    \SetWidth{3.0}
    \SetColor{Black}
    \Arc(224,-142)(48,180,540)
    \Line(192,-110)(192,-174)
    \Line(240,-94)(240,-190)
    \Line(272,-142)(336,-142)
    \Line(176,-142)(112,-142)
\end{picture}
}}\,}

\def\graphedeltasansindiced{\,{\scalebox{0.20}{
\begin{picture}(124,52) (201,-165)
    \SetWidth{3.0}
    \SetColor{Black}
    \Arc(280,-154)(36.878,49.399,282.529)
    \Line(240,-158)(192,-158)
    \Line(256,-126)(256,-190)
\end{picture}
}}\,}
\def\graphedeltacontrau{\,{\scalebox{0.20}{
\begin{picture}(194,46) (191,-177)
    \SetWidth{3.0}
    \SetColor{Black}
    \Arc(288,-164)(32,180,540)
    \Line(320,-164)(384,-164)
    \Line(256,-164)(192,-164)
\end{picture}
}}\,}
\def\graphedeltacontrat{\,{\scalebox{0.20}{
\begin{picture}(173,70) (178,-180)
    \SetWidth{3.0}
    \SetColor{Black}
    \Arc(268,-165)(34.482,210,570)
    \Line(303,-171)(350,-171)
    \Line(233,-171)(179,-171)
    \Line(280,-133)(280,-197)
\end{picture}
  }}\,}

\def\graphedeltacontratddd{\,{\scalebox{0.20}{
\begin{picture}(96,57) (336,-297)
    \SetWidth{3.0}
    \SetColor{Black}
\Arc(382.723,-358.862)(114.921,68.415,112.902)
    \Arc(379,-191.833)(131.078,-108.688,-67.576)
    \Line(354,-248)(354,-320)
    \Line(412,-249)(415,-317)
\end{picture}
  }}\,}
\def\graphedeltacontratdddd{\,{\scalebox{0.20}{
\begin{picture}(168,29) (184,-180)
    \SetWidth{3.0}
    \SetColor{Black}
    \Arc(243,-176)(22.627,135,495)
    \Arc(288,-175)(22.627,135,495)
    \Line(312,-176)(351,-176)
    \Line(221,-177)(185,-177)
\end{picture}
  }}\,}

\def\graphedeltajdid{\,{\scalebox{0.20}{
\begin{picture}(153,66) (124,-279)
    \SetWidth{3.0}
    \SetColor{Black}
    \Arc[clock](250.949,-261.153)(42.881,-72.281,-290.548)
    \Line(208,-263)(155,-264)
    \Line(225,-230)(225,-293)
    \Line(252,-218)(252,-304)
\end{picture}}}\,}
\def\graphenonpi{\,{\scalebox{0.20}{
\begin{picture}(178,34) (143,-143)
    \SetWidth{3.0}
    \SetColor{Black}
    \Line(144,-126)(192,-126)
    \Arc(208,-126)(16,90,450)
    \Line(224,-126)(256,-126)
    \Arc(272,-126)(16,90,450)
    \Line(288,-126)(320,-126)
  \end{picture}}}\,}
 \def\qedj{\,{\scalebox{0.20}{ \begin{picture}(181,66) (218,-178)
    \SetWidth{3.0}
    \SetColor{Black}
    \PhotonArc(307.786,-179.786)(22.955,14.598,173.03){7.5}{3.5}
    \PhotonArc[clock](306,-170.142)(49.152,-175.498,-364.502){7.5}{8.5}
    \Line(219,-175)(398,-175)
  \end{picture}}}\,}
\def\qeda{\,{\scalebox{0.20}{
\begin{picture}(98,35) (271,-286)
    \SetWidth{3.0}
    \SetColor{Black}
    \PhotonArc(320,-277)(17.889,-26.565,206.565){7.5}{4.5}
    \Line(272,-285)(368,-285)
\end{picture}}}\,}

\def\qedcoiz{\,{\scalebox{0.20}
  { \begin{picture}(258,93) (191,-180)
    \SetWidth{3.0}
    \SetColor{Black}
    \PhotonArc(328,-187)(91.214,15.255,164.745){7.5}{12.5}
    \Line(192,-163)(448,-163)
    \Line(335.998,-178.998)(304.002,-147.002)\Line(304.002,-178.998)(335.998,-147.002)
    \Text(315,-210)[lb]{\Huge{\Black{$0$}}}
  \end{picture}}}\,}
\def\qedcoiu{\,{\scalebox{0.20}
  { \begin{picture}(258,93) (191,-180)
    \SetWidth{3.0}
    \SetColor{Black}
    \PhotonArc(328,-187)(91.214,15.255,164.745){7.5}{12.5}
    \Line(192,-163)(448,-163)
    \Line(335.998,-178.998)(304.002,-147.002)\Line(304.002,-178.998)(335.998,-147.002)
    \Text(315,-210)[lb]{\Huge{\Black{$1$}}}
  \end{picture}}}\,}
\def\grpqed{\,{\scalebox{0.20}{
\begin{picture}(194,74) (175,-160)
    \SetWidth{3.0}
    \SetColor{Black}
    \Arc(272,-150)(35.777,117,477)
    \Photon(176,-150)(240,-150){7.5}{3}
    \Photon(304,-150)(368,-150){7.5}{3}
    \Photon(256,-118)(256,-182){7.5}{3}
    \Photon(288,-118)(288,-182){7.5}{3}
\end{picture}}}\,}
\def\grpqeda{\,{\scalebox{0.20}{
  \begin{picture}(144,94) (200,-155)
    \SetWidth{3.0}
    \SetColor{Black}
    \Photon(201,-131)(257,-131){7.5}{3}
    \Line(256,-131)(343,-85)
    \Line(257,-133)(336,-177)
    \Photon(286,-116)(291,-149){7.5}{2}
    \Photon(321,-97)(320,-168){7.5}{4}
\end{picture}}}\,}
\def\grpqedb{\,{\scalebox{0.20}{
\begin{picture}(129,82) (191,-170)
    \SetWidth{3.0}
    \SetColor{Black}
    \Photon(192,-160)(257,-160){7.5}{3}
    \Line(257,-161)(319,-126)
    \Line(257,-160)(318,-206)
    \Photon(299,-139)(302,-190){7.5}{3}   
\end{picture}}}\,}
\def\grpqedc{\,{\scalebox{0.20}{
 \begin{picture}(134,52) (193,-148)
    \SetWidth{3.0}
    \SetColor{Black}
    \Arc(258,-143)(24.597,153,513)
    \Photon(194,-140)(235,-141){7.5}{2}
    \Photon(284,-142)(326,-139){7.5}{2}
\end{picture}}}\,}
\def\grpqedd{\,{\scalebox{0.20}{
\begin{picture}(128,48) (255,-198)
    \SetWidth{3.0}
    \SetColor{Black}
    \Arc(320,-195)(22.825,151,511)
    \Photon(321,-174)(322,-218){7.5}{2}
    \Photon(345,-196)(382,-193){7.5}{2}
    \Photon(297,-193)(256,-193){7.5}{2}
  \end{picture}}}\,}
\def\qedin{\,{\scalebox{0.20}{
   \begin{picture}(194,40) (127,-200)
    \SetWidth{3.0}
    \SetColor{Black}
    \Photon(128,-190)(192,-190){7.5}{3}
    \Photon(256,-190)(320,-190){7.5}{3}
    \Arc(224,-190)(32,270,630)
    \Line(175.998,-205.998)(144.002,-174.002)\Line(144.002,-205.998)(175.998,-174.002)
  \end{picture}}}\,}
\def\qedrefz{\,{\scalebox{0.20}{
  \begin{picture}(143,40) (201,-200)
    \SetWidth{3.0}
    \SetColor{Black}
    \Arc(273,-186)(22.204,144,504)
    \Line(269.001,-213.999)(278.999,-202.001)\Line(268.001,-203.001)(279.999,-212.999)
    \Text(276,-242)[lb]{\Huge{\Black{$0$}}}
    \Photon(202,-194)(250,-191){7.5}{2}
    \Photon(296,-190)(343,-191){7.5}{2}
  \end{picture}}}\,}
  \def\qedrefu{\,{\scalebox{0.20}{
    \begin{picture}(143,40) (201,-200)
    \SetWidth{3.0}
    \SetColor{Black}
    \Arc(273,-186)(22.204,144,504)
    \Line(269.001,-213.999)(278.999,-202.001)\Line(268.001,-203.001)(279.999,-212.999)
    \Text(276,-242)[lb]{\Huge{\Black{$1$}}}
    \Photon(202,-194)(250,-191){7.5}{2}
    \Photon(296,-190)(343,-191){7.5}{2}
  \end{picture}}}\,}

\def\diagramme #1{\vskip 4mm \centerline {#1} \vskip 4mm}
\begin{document}
\title{
{Bialgebra of specified graphs and external structures}}

\author{Dominique Manchon}
\address{Universit\'e Blaise Pascal,
       C.N.R.S.-UMR 6620,
        63177 Aubi\`ere, France}       
        \email{manchon@math.univ-bpclermont.fr}
        \urladdr{http://math.univ-bpclermont.fr/~manchon/}
         
\author{Mohamed Belhaj Mohamed}
\address{{Universit\'e Blaise Pascal,
         laboratoire de math\'ematiques UMR 6620,
         63177 Aubi\`ere, France.}\\
         {Laboratoire de math\'ematiques physique fonctions sp\'eciales et applications, universit\'e de sousse, rue Lamine Abassi 4011 H. Sousse,  Tunisie}}     
         \email{Mohamed.Belhaj@math.univ-bpclermont.fr}

\date{June, 17th 2013}
\noindent{\footnotesize{${}\phantom{a}$ }}
\begin{abstract}
We construct a Hopf algebra structure on the space of specified Feynman graphs of a quantum field theory. We introduce a convolution product and a semigroup of characters of this Hopf algebra with values in some suitable commutative algebra taking momenta into account. We then implement the renormalization described by A. Connes and D. Kreimer in \cite{A.D2000} and the Birkhoff decomposition for two renormalization schemes: the minimal subtraction scheme and the Taylor expansion scheme.
\end{abstract}
\maketitle

\tableofcontents
\section{Introduction}
Hopf algebras of Feynman graphs have been studied by A. Connes and D. Kreimer in \cite{A.D2000}, \cite{ad98}, \cite{ad01} and \cite{dk98} as a powerful tool to explain the combinatorics of renormalization in quantum field theory. In this note we are interested in the Hopf algebra of specified Feynman graphs studied by A. Connes and D. Kreimer in \cite{A.D2000}.\\

In the first part we study the simpler case of Hopf algebras of locally one-particle irreducible ($1PI$) Feynman graphs, neglecting the specification at this stage. First we consider a theory of fields $\Cal T$ (for example $\varphi^3$ \cite{A.D2000}, $\varphi^4$ \cite{wvs}, $QED$ and $QCD$ \cite{wvs} , \cite{wvs06}...) which gives rise to Feynman graph types determined by $\Cal T$: the type of a vertex is determined by the type of half-edges that are adjacent to it. We then construct a structure of commutative bialgebra  $\wt{\Cal H}_{\Cal T}$ on the space of locally $1PI$ graphs of $\Cal T$. The coproduct is given by:
\begin{eqnarray*}
\Delta (\Gamma ) &=& \sum_{\substack{\gamma \subseteq \Gamma \\ \Gamma / \gamma \in \Cal T }}  \gamma \otimes \Gamma / \gamma,
\end{eqnarray*}
where the sum runs over all locally $1PI$ covering subgraphs of $\Gamma$ such that the contracted subgraph  $\Gamma/\gamma$ is in the theory $\Cal T$ (in other words, locally $1PI$ superficially divergent subgraphs \cite{SA}). The Hopf algebra ${\Cal H}_{\Cal T}$ is obtained by taking the quotient of the bialgebra $\wt{\Cal H}_{\Cal T}$ above by the ideal generated by $\un - \Gamma$, where the unit $\un$ is the empty graph and $\Gamma$ is a $1PI$ graph without internal edges. Then we introduce the specification: we are led by quantum field theory to distinguish between vertices of the same type. For example the list of vertices admitted in $\varphi^3$ theory and in $QED$ are respectively:
\vskip-1cm
\begin{eqnarray*}
\{ \graphearetenorcroixzero , \graphearetenorcroixun , \graphev  \} \;\;\;\text{and} \;\;\; \{ \graphearetenorcroixzero , \graphearetenorcroixun , \graphevcur , \graphearetecurcroixun \}.
\end{eqnarray*}

The contraction of a subgraph on a point rises a problem: For example in $\varphi^3$ theory, if we contract the subgraph $\graphecerc$ inside the graph $\grapheexmdel$, shall we get $\graphecerccroizero$ or $\graphecerccroiun$ ?\\

\hskip-2mmSimilarly for $QED$, does the contraction of $\qeda$ inside $\qedj$ give $\qedcoiz$ or $\qedcoiu$?
\vskip3mm
We will remedy to this by introducing the specified graph $\bar\Gamma = (\Gamma, \underline{i})$ where $\underline{i}$ is a multi-index which identifies the type of residue of $\Gamma$, that is to say, the vertices obtained by contracting each connected components onto a point. (For example $\mop {res} \graphedeltasansindiced = \graphev$). The formula for the coproduct of specified graphs is then given by: 
\begin{eqnarray*}
\Delta (\bar\Gamma ) &=& \sum_{{\bar\gamma \subseteq \bar\Gamma \atop \bar\Gamma / \bar\gamma \in \Cal T }} \bar \gamma \otimes \bar\Gamma / \bar\gamma,
\end{eqnarray*}
where the sum runs over all locally $1PI$ specified covering subgraphs $ \bar\gamma = (\gamma,\underline{j})$ of $ \bar\Gamma = (\Gamma,\underline{i})$ (see definition \ref{df1}), such that the contracted subgraph $(\Gamma/(\gamma,\underline{j}) , \underline{i})$ is in the theory $\Cal T$. Here $\underline{j}$ is a multi-index that identifies the residue of each of the connected components of $\gamma$. The Hopf algebra ${\Cal H}_{\Cal T}$ is again obtained by identifying the specified graphs without internal edges with the unit.\\

In the second part we are interested in external structures. Feynman rules associate to each graph a function which depends on moments associated with each  half edge of the graph, with the constraints $p_e + p_{e'} = 0$ for each internal edge $(e e')$ and $\sum_{e \in st(v)} p_e = 0$ for any vertex $v$, where $st(v)$ is the set of half-edges adjacent to $v$. Feynman rules $\Phi$ do depend on the refined types of vertices, but do not depend on the overall specification. Refined types of vertices combine themselves in a nice way: for example, in $\varphi^3$ theory, $\graphearetenorcroixzero$ and $\graphearetenorcroixun$ combine to an edge $\graphearetenor$, namely \cite{A.D2000}:
$$\Phi (\graphecerccroizero) + \Phi (\graphecerccroiun) = \Phi (\graphedeltacontrau).$$
Similary in $QED$ we have: 
$$\Phi (\qedin) = \Phi (\grpqedc),$$
and
$$\Phi (\qedrefz) + \Phi (\qedrefu) = \Phi (\grpqedc).$$
Considering the relations above we could have chosen only one type of bivalent vertex for $\varphi^3$ or for the electron edges in $QED$, and discard the bivalent vertex for the photon edges: these are the conventions adopted in \cite{wvs06}. We have chosen not to consider this simplification, in order to follow \cite{A.D2000} more closely.\\

We introduce a semi-group $G$ of characters of $\wt{\Cal H}_{\Cal T}$ with values in some suitable commutative algebra ${\Cal B}$, and a convolution product $\circledast$ on G. We then implement the renormalization described by A. Connes and D. Kreimer in \cite{A.D2000} (see also \cite[Chap. 1 \S\ 5 \& 6]{CM}), replacing ${\Cal B}$ by $\Cal A := \Cal B [z^{-1} , z]]$. We show that each element of $G$ has a unique Birkhoff decomposition for minimal renormalization scheme ${\Cal A} = {\Cal A}_- \oplus {\Cal A}_+$, where $ {\Cal A}_+ : =   \Cal B [[ z ]]$ and $ {\Cal A}_- : =  z^{-1} \Cal B [z^{-1}]$. We also implement the Birkhoff decomposition associated with Taylor expansions in the algebra ${\Cal B}$ itself, along the lines of \cite{EP}. The interest of the construction presented here is the purely combinatorial nature of the bialgebra $\wt{\Cal H}_{\Cal T}$ and the Hopf algebra ${\Cal H}_{\Cal T}$: all the dependence on momenta is removed in the target algebra ${\Cal B}$. The Feynman rules, given by the integration of these functions on the internal momenta, will be the subject of a future article.\\

\noindent
{\bf Acknowledgements :} We would likes to thank Kurusch Ebrahimi-Fard for discussion and remarks. Research supported by projet CMCU Utique 12G1502 and by Agence Nationale de la Recherche (CARMA ANR-12-BS01-0017).
\section{Hopf algebras of Feynman graphs}
\subsection{Basic definitions}
A Feynman graph is a graph (non-planar) with a finite number of vertices and edges, which can be internal or external. An internal edge is an edge connected at both ends to a vertex, an external edge is an edge with one open end, the other end being connected to a vertex. The edges are obtained by using a half-edges.\\
More precisely, let us consider two finite sets $\Cal V$ and $ \Cal E$. A graph $\Gamma$ with $\Cal V$ (resp. $\Cal E$) as set of vertices (resp. half-edges) is defined as follows: let $ \sigma : \Cal E \longrightarrow \Cal E $ be an involution and $\partial : \Cal E  \longrightarrow \Cal V$. For any vertex $v\in \Cal V$ we denote by $st(v) = \{ e \in \Cal E / \partial (e) = v \}$ the set of half-edges adjacent to $v$. The fixed points of $\sigma$ are the \textsl {external edges} and the \textsl {internal edges} are given by the pairs $\{ e , \sigma (e)\}$ for $ e \neq \sigma (e)$. The graph $\Gamma$ associated to these data is obtained by attaching  half-edges $e\in st(v)$ to any vertex $v\in\Cal V$, and joining the two half-edges $e$ and $\sigma(e)$ if $\sigma(e)\not =e$.\\

Several types of half-edges will be considered later on: the set $\Cal E$ is partitioned into several pieces $\Cal E_i$. In that case we ask that the involution $\sigma$ respects the different types of half-edge, i.e. $\sigma(\Cal E_i)\subset \Cal E_i$.\\

We denote by $\Cal I(\Gamma)$ the set of internal edges and by $\mop{Ext}(\Gamma)$ the set of external edges. The \textsl {loop number} of a graph $\Gamma$ is given by: 
$$L(\Gamma) =  \left|\Cal I (\Gamma)\right| - \left|\Cal V (\Gamma)\right|  + \left|\pi_0 (\Gamma)\right|,$$
where $\pi_0 (\Gamma)$ is the set of connected components of $\Gamma$.\\

A one-particle irreducible graph (in short, $1PI$ graph) is a connected graph which remains connected when we cut any internal edge. A disconnected graph is said to be locally $1PI$ if any of its connected components is $1PI$.
$$\graphedeltacontrat \;\;\;\; \text{is}\;\;\; 1PI  \;\;\;\;\;\;\; \text {and} \;\;\; \graphenonpi \;\;\; \text {is not} \;\;\; 1PI.$$

A covering subgraph of $\Gamma$ is a Feynman graph $\gamma$ (not necessarily connected), obtained from $\Gamma$ by cutting internal edges. In other words:
\begin{enumerate}
\item $ \Cal V (\gamma) = \Cal V (\Gamma)$. 
\item $ \Cal E (\gamma) =  \Cal E (\Gamma)$.
\item $ \sigma_\Gamma (e) = e \Longrightarrow \sigma_\gamma (e) = e $.
\item If $ \sigma_\gamma (e) \neq  \sigma_\Gamma (e) \;\; \text{then} \;\; \sigma_\gamma (e) = e \;\; \text{and} \;\; \sigma_\gamma (\sigma_\Gamma (e)) = \sigma_\Gamma (e)$.
\end{enumerate}

For any covering subgraph $\gamma$, the contracted graph $\Gamma / \gamma$ is defined by shrinking all connected components of $\gamma$ inside $\Gamma$ onto a point. 
$$ \Gamma = \graphedeltasansindice  \;\;\; ,\;\;  \gamma = \graphedeltasansindiced \graphev \graphev \graphev \;\; \Longrightarrow  \;\; \Gamma / \gamma = \graphedeltacontrat$$
$$ \Gamma = \graphedeltasansindice  \;\;\; , \;\; \gamma = \graphedeltasansindiced \graphedeltasansindiced \;\; \Longrightarrow  \;\; \Gamma / \gamma = \graphedeltacontrau $$
 
The residue of the graph $\Gamma$, denoted by $\mop{res} \Gamma$, is the contracted graph $\Gamma / \Gamma$. In other words: it is the only graph with no internal edge and the same external edges than those of $\Gamma$.
$$ \mop{res} (\graphedeltasansindice) =  \graphearetenorcroix \;\;\text {et}\;\; \mop{res} (\graphedeltasansindiced) = \graphev.$$

The skeleton of a graph $\Gamma$ denoted by $\mop {sk} \Gamma$ is a graph obtained by cutting all internal edges, for example:
 $$\mop{sk} (\grapheexmdel) = \graphev \graphev \graphev \graphev.$$
   
\subsection {Bialgebra $\wt{\Cal H}_{\Cal T}$} 
We will work inside a physical theory $\Cal T$, which involves Feynman graphs of some prescribed type: $\varphi^3 , \; \varphi^4$, QED, QCD ...\\
We denote by $\Cal E (\Cal T)$ the set of possible types of half-edges and by $\Cal V(\Cal T)$ the set of possible types of vertices.
\begin{ex}
$\Cal E ( \varphi^3) = \{ \graphearetenor \} \;\;\;, \;\;\; \Cal V ( \varphi^3) = \{\graphearetenorcroix , \graphev \}$\\
$\Cal E ( QED) = \{ \graphearetenor  , \graphearetecur \} \;\;\; , \;\;\; \Cal V ( QED) = \{\graphevcur , \graphearetenorcroix , \graphearetecurcroix  \}$
\end{ex}
An element of $\Cal V (\Cal T)$ can be seen as a function from $\Cal E (\Cal T)$ into $\mathbb N$ which to each type of half-edge associates the number of half-edges of that type arriving on the vertex in question. Actually the typology of vertices presented here is too coarse, we will return to this point in Section 2, with the introduction of specified graphs. \\

Let $ \wt V_{\Cal T}$ be the set of $1PI$ connected graphs $\Gamma$ with edges in $\Cal E (\Cal T)$ and vertices in $\Cal V(\Cal T)$ such that $\mop {res} \Gamma$ is a vertex in $\Cal V_{\Cal T} $ (condition of superficial divergence \cite{SA}, \cite{A.D2000}, \cite{dk98}). Let $\wt {\Cal H}_{\Cal T} = S ( \wt V_{\Cal T}) $ be the vector space generated by superficially divergent, locally $1PI$, not necessarily connected Feynman graphs. The product is given by concatenation, the unit $\un$ is identified to the empty graph, and the coproduct is defined by :
\begin{eqnarray*}
\Delta (\Gamma ) &=& \sum_{{\gamma \subseteq \Gamma \atop \Gamma / \gamma \in \Cal T }}  \gamma \otimes \Gamma / \gamma.
\end{eqnarray*}
In the above sum, $\gamma$ runs over all locally $1PI$ covering subgraphs of $\Gamma$ such that the contracted subgraph  $\Gamma/\gamma$ is in the theory $\Cal T$.
\begin{ex} In $\varphi^3$ Theory:
\begin{eqnarray*}
\Delta (\graphedeltasansindice ) &=& \graphev \graphev \graphev \graphev \graphev \graphev \otimes \graphedeltasansindice + \graphedeltasansindice \otimes \graphearetenorcroix\\
&&\\
&+& 2 \graphedeltasansindiced \graphev \graphev \graphev \otimes \graphedeltacontrat + 2 \graphedeltajdid \graphev \otimes \graphedeltacontrau \\ 
&&\\
&+& \graphedeltasansindiced\graphedeltasansindiced\otimes\graphedeltacontrau + \graphedeltacontratddd \graphev \graphev\otimes \graphedeltacontratdddd \hskip-45mm \graphecroi.
\end{eqnarray*}
The last term is removed because $\graphedeltacontratdddd \notin \varphi^3$.\\
In $Q E D$:
\begin{eqnarray*}
\Delta (\grpqed) &=& \graphevcur \graphevcur \graphevcur \graphevcur \graphevcur \graphevcur  \otimes \grpqed \\
&&\\
&+& \grpqed \otimes \graphearetecurcroix + 2 \grpqeda \graphevcur \otimes \grpqedc\\
&&\\ 
&+& 2 \grpqedb \graphevcur\graphevcur\graphevcur \otimes \grpqedd + \grpqedb  \grpqedb \otimes\grpqedc.
\end{eqnarray*}
\end{ex}
\begin{theorem}
Equipped with this coproduct, $\wt {\Cal H}_{\Cal T}$ is a bialgebra. 
\end{theorem} 
\begin{proof}
$\Delta $ is coassociative. Indeed: 
\begin{eqnarray*}
(\Delta \otimes id)\Delta (\Gamma ) &=& \sum_{{\gamma \subseteq \Gamma \atop \Gamma / \gamma \in \Cal T }}  \Delta(\gamma) \otimes \Gamma / \gamma\\
&=&\sum_{{\delta \subseteq \gamma \subseteq \Gamma \atop \gamma / \delta \;;\; \Gamma / \gamma \in \Cal T }}  \delta \otimes \gamma / \delta \otimes \Gamma / \gamma.
\end{eqnarray*}
\begin{eqnarray*}
(id \otimes \Delta)\Delta (\Gamma ) &=& \sum_{{\delta \subseteq \Gamma \atop \Gamma / \delta \in \Cal T }}  \delta \otimes \Delta(\Gamma / \delta)\\
&=&\sum_{{\delta \subseteq \Gamma \;;\wt\gamma \subseteq \Gamma / \delta  \atop (\Gamma / \delta)/\wt\gamma \;;\; \Gamma / \delta \in \Cal T }}  \delta \otimes \wt\gamma  \otimes (\Gamma / \delta)/\wt\gamma.
\end{eqnarray*}
For any covering subgraph $\delta$ of $\Gamma$ such that $ \Gamma/ \delta\in\Cal T$, there is an abvious bijection $\gamma \mapsto \wt\gamma = \gamma / \delta$ from covering subgraphs of $\Gamma$ containing $ \delta$ such that $ \Gamma / \gamma \in \Cal T$ and $ \gamma / \delta \in \Cal T$, onto covering subgraphs of $ \Gamma / \delta$ such that $(\Gamma / \delta)/\wt\gamma \in \Cal T$, given by shriking $\delta$ \cite{Dm11}. For all $\wt\gamma \subseteq \Gamma / \delta$ there exist a unique covering subgraph $\gamma$ of $\Gamma$ containing $\delta$ such that: $ \wt\gamma \cong \gamma / \delta$ and we have: $ (\Gamma/ \delta) / \wt\gamma \cong \Gamma / \gamma$. We then obtain:
\begin{eqnarray*}
(id \otimes \Delta)\Delta (\Gamma ) &=& \sum_{{\delta \subseteq \gamma \subseteq \Gamma \atop \Gamma / \delta  \in \Cal T }}  \delta \otimes \Delta(\Gamma / \delta)\\
&=&\sum_{{\delta \subseteq \gamma \subseteq \Gamma  \atop \Gamma / \gamma \;;\; \gamma / \delta \in \Cal T }}  \delta \otimes \gamma / \delta \otimes \Gamma / \gamma.
\end{eqnarray*}
The two expressions coincide, therefore $\Delta$ is coassociative. The counit is given by: $ \varepsilon (\Gamma) = 1$ if $\Gamma$ has no internal edges, and $\varepsilon ( \Gamma) = 0$ for any graph having at least one internal edge. The bialgebra $\wt {\Cal H}_{\Cal T}$ is graded, and the grading is given by the number $L$.
\end{proof}
\subsection { Hopf algebra ${\Cal H}_{\Cal T}$}
The Hopf algebra ${\Cal H}_{\Cal T}$ is given by  identifying all elements of  degree zero (the residues) to unit $\un$:
\begin{equation}
{\Cal H}_{\Cal T} = \wt{\Cal H}_{\Cal T} / \Cal J
\end{equation}
where $\Cal J$  is the ideal generated by the elements $\un - \mop{res} \Gamma$ where $\Gamma$ is an $1PI$ graph. ${\Cal H}_{\Cal T}$ is a connected graded bialgebra, it is therefore a connected graded Hopf algebra, which can be identified as a commutative algebra with $S(\Cal V_\Cal T)$, where $\Cal V_\Cal T$ is the vector space generated by the $1PI$ connected Feynman graphs. The coproduct then becomes:
\begin {equation}
\Delta (\Gamma ) = \un \otimes \Gamma + \Gamma \otimes \un + \sum_{{\gamma \hbox{ \sevenrm proper subgraph of } \Gamma \atop \hbox{ \sevenrm loc 1PI. }\; \Gamma / \gamma \in \Cal T }} \gamma \otimes \Gamma / \gamma.
\end{equation}
\begin{ex} In $\varphi^3$ Theory:
\begin{eqnarray*}
\Delta (\graphedeltasansindice ) &=& \un \otimes \graphedeltasansindice + \graphedeltasansindice \otimes \un + 2 \graphedeltasansindiced \otimes \graphedeltacontrat\\
&&\\
&+& 2 \graphedeltajdid \otimes \graphedeltacontrau + \graphedeltasansindiced\graphedeltasansindiced\otimes\graphedeltacontrau + \graphedeltacontratddd \otimes \graphedeltacontratdddd \hskip-30mm \graphecroipetit.  
\end{eqnarray*}
In $Q E D$:
\begin{eqnarray*}
\Delta (\grpqed) &=& \un  \otimes \grpqed + \grpqed \otimes \un + \grpqedb  \grpqedb \otimes\grpqedc\\
&&\\
&+&  2 \grpqeda \otimes \grpqedc + 2 \grpqedb \otimes \grpqedd.
\end{eqnarray*}

\end{ex}
\section{Specified Feynman graphs Hopf Algebra}
\subsection {Bialgebra $\wt{\Cal H}_{\Cal T}$}
In this paragraph, we denote by  $\wt{\Cal V}(\Cal T)$ the set of possible refined types of vertices: for $ t \in {\Cal V}(\Cal T)$, you can have a vertices of the same type $t$ but with different refined type. For any refined type $\wt t \in \wt{\Cal V}(\Cal T)$ we denote by $[\; \wt {t} \;]$ the underlying vertex type. We denote also  $\wt t= (t,i)$ where the index $i$ serves to distinguish the refined types of same underlying type.
\begin{ex}
$\wt{\Cal V}(\varphi^3) = \{ \graphearetenorcroixzero , \graphearetenorcroixun , \graphev  \}$ \\
$\wt{\Cal V}(QED) = \{ \graphearetenorcroixzero , \graphearetenorcroixun , \graphevcur , \graphearetecurcroixun \}$ 
\end{ex}
\begin{remark}
Note that the types of half-edges adjacent to a vertex $v$ are not sufficient to determine its refined type.
\end{remark}
\begin{definition} \label{df1}
A specified graph of theory $\Cal T$ is a couple $(\Gamma , \underline{i})$ where:
\begin{enumerate}
\item $\Gamma$ is a locally $1PI$ superficially divergent graph with half-edges and vertices of the type prescribed in $\Cal T$. 
\item $\underline{i} : \pi_0(\Gamma) \longrightarrow \mathbb N$, the values of $\underline {i}(\gamma)$ being prescribed by the possible types of vertex obtained by contracting the connected component $\gamma$ on a point.
\end{enumerate}
We will say that $(\gamma , \underline{j})$ is a specified covering subgraph of $(\Gamma , \underline{i})$, $\big( (\gamma , \underline{j}) \subset (\Gamma , \underline{i}) \big)$  if:
\begin{enumerate}
\item $\gamma$ is a covering subgraph of $\Gamma$.
\item if $\gamma_0$ is a full connected component of $\gamma$, i.e if $\gamma_0$ is also a full connected component of $\Gamma$, then $\underline{j} (\gamma_0) = \underline{i} (\gamma_0)$.
\end{enumerate}
\end{definition}
\begin{remark}
Sometimes we denote by $\bar \Gamma = (\Gamma , \underline{i})$ the specified graph, and we will write $\bar \gamma \subset \bar \Gamma$ for $(\gamma , \underline{j}) \subset (\Gamma , \underline{i})$.
\end{remark}
\begin{definition}
Let be $(\gamma , \underline{j}) \subset (\Gamma , \underline{i})$. The contracted specified subgraph is written:
$$ \bar\Gamma / \bar \gamma = (\Gamma/\bar\gamma , \underline{i}), $$
where $\bar\Gamma/\bar\gamma$ is obtained by contracting each connected component of $\gamma$ on a point, and specifying the vertex obtained with $\underline {j}$.
\end{definition}
\begin{remark}
The specification $\underline{i}$ is the same for the graph $\bar\Gamma$ and the contracted graph $\bar\Gamma / \bar \gamma$. 
\end{remark}

Let $\wt {\Cal H}_{\Cal T}$ be the vector space generated by the specified superficially divergent Feynman graphs of a field theory $\Cal T$. The product is given by the concatenation, the unit $ \un $ is identified with the empty graph and the coproduct is defined by:
\begin{eqnarray*}
\Delta (\bar\Gamma ) &=& \sum_{{\bar\gamma \subseteq \bar\Gamma \atop \bar\Gamma / \bar\gamma \in \Cal T }} \bar \gamma \otimes \bar\Gamma / \bar\gamma,
\end{eqnarray*}
where the sum runs over all locally $1PI$ specified covering subgraphs $ \bar\gamma = (\gamma,\underline{j})$ of $ \bar\Gamma = (\Gamma,\underline{i})$ , such that the contracted subgraph $(\Gamma/(\gamma,\underline{j}) , \underline{i})$ is in the theory $\Cal T$.
\begin{theorem} 
Equipped with the coproduct $\Delta$,  $\wt {\Cal H}_{\Cal T}$ is a bialgebra.
\end{theorem}
\begin{proof}
$\Delta $ is coassociative. Indeed : 
\begin{eqnarray*}
(\Delta \otimes id)\Delta (\bar\Gamma ) &=& \sum_{{\bar\gamma \subseteq \bar\Gamma \atop \bar\Gamma / \bar\gamma \in \Cal T }}  \Delta(\bar\gamma) \otimes \bar\Gamma / \bar\gamma\\
&=&\sum_{{\bar\delta \subseteq \bar\gamma \subseteq \bar\Gamma \atop \bar\gamma / \bar\delta \;\;;\; \bar\Gamma / \bar\gamma \in \Cal T }}  \bar\delta \otimes \bar\gamma / \bar\delta \otimes \bar\Gamma / \bar\gamma.
\end{eqnarray*}
\begin{eqnarray*}
(id \otimes \Delta)\Delta (\bar\Gamma ) &=& \sum_{{\bar\delta \subseteq \bar\Gamma \atop \bar\Gamma / \bar\delta \in \Cal T }}  \bar\delta \otimes \Delta(\bar\Gamma / \bar\delta)\\
&=&\sum_{{\bar\delta \subseteq \bar\Gamma \;\;;\bar\alpha \subseteq \bar\Gamma / \bar\delta  \atop (\bar\Gamma / \bar\delta)/\bar\alpha \;\;;\; \bar\Gamma / \bar\delta \in \Cal T }}  \bar\delta \otimes \bar\alpha  \otimes (\bar\Gamma / \bar\delta)/\bar\alpha.
\end{eqnarray*}
For any specified covering subgraph $\bar\delta$ of $\Gamma$ such that $ \bar\Gamma/ \bar\delta\in\Cal T$, there is an obvious bijection $\bar\gamma \mapsto \bar\alpha = \bar\gamma / \bar\delta$ from specified covering subgraphs of $\bar\Gamma$ containing $ \bar\delta$ such that $(\bar\Gamma /\bar\delta)/\bar\alpha ; \bar\Gamma / \bar\delta \in \Cal T$, onto specified covering subgraphs of $\bar\Gamma / \bar\delta$ such that $\bar\gamma / \bar\delta \;\;; \bar\Gamma / \bar\gamma \in \Cal T$, given by shriking $\bar\delta$. Then for any specified covering subgraph $\bar\alpha = ( \alpha , \underline j)$ of $\Gamma / \bar\delta$ there exists a unique specified covering subgraph $\bar\gamma = (\gamma , \underline j)$ of $\bar\Gamma$ such that $\delta \subseteq \gamma $ and  $ \alpha \cong \gamma /\bar\delta$, we have: $  \bar\alpha \cong \bar\gamma / \bar\delta$  and $ (\bar\Gamma/ \bar\delta) / \bar\alpha \cong \bar\Gamma / \bar\gamma$. We then obtain:
\begin{eqnarray*}
(id \otimes \Delta)\Delta (\bar\Gamma ) &=& \sum_{{\bar\delta \subseteq \bar\gamma \subseteq \bar\Gamma \atop \bar\gamma / \bar\delta \;\;;\; \bar\Gamma / \bar\gamma \in \Cal T }}  \bar\delta \otimes \bar\gamma / \bar\delta \otimes \bar\Gamma / \bar\gamma.
\end{eqnarray*}
Then $\Delta$ is coassociative, $\wt {\Cal H}_{\Cal T}$ is a bialgebra where the counit is given by $ \varepsilon (\bar\Gamma) = 1$ if $\bar\Gamma$ has no internal edges, and $\varepsilon ( \bar\Gamma) = 0$ for any graph $\bar\Gamma$ having at least one internal edge.
\end{proof}
\begin{ex} \label{ex1}
In $\varphi^3$ Theory:
\begin{eqnarray*}
\Delta(\grapheexmdel , 0) &=& (\grapheexmdel , 0) \otimes \graphearetenorcroixzero  + \graphev \graphev \graphev \graphev \otimes (\grapheexmdel , 0)\\
&&\\
&+& (\graphev \graphev \graphecerc , 0)\otimes (\graphecerccroizero , 0)\\
&&\\
&+&  (\graphev \graphev \graphecerc , 1)\otimes (\graphecerccroiun , 0).
\end{eqnarray*}
In $QED$:
\begin{eqnarray*}
\Delta (\qedj , 1) &=& \graphevcur \graphevcur \graphevcur \graphevcur\otimes (\qedj , 1)\\
&&\\
&+& (\qedj , 1) \otimes \graphearetenorcroixun  + (\qeda \graphevcur \graphevcur, 0) \otimes (\qedcoiz , 1)\\
&&\\ 
&+&(\qeda \graphevcur \graphevcur, 1) \otimes (\qedcoiu , 1).
\end{eqnarray*}
\end{ex}
\subsection {Hopf algebra ${\Cal H}_{\Cal T}$}
The Hopf algebra ${\Cal H}_{\Cal T}$ is given by  identifying all elements of degree zero (the residues) to unit $\un$:
\begin{equation}
{\Cal H}_{\Cal T} = \wt{\Cal H}_{\Cal T} / \Cal J
\end{equation}
where $\Cal J$  is the ideal generated by the elements $\un - \mop{res} \bar\Gamma$ where $\bar\Gamma$ is an $1PI$ specified graph. ${\Cal H}_{\Cal T}$ is a connected graded bialgebra, it is therefore a connected graded Hopf algebra, the coproduct then becomes:

\begin {equation}
\Delta (\bar\Gamma ) = \un \otimes \bar\Gamma + \bar\Gamma \otimes \un + \sum_{{\bar\gamma \hbox{ \sevenrm proper subgraph of }\; \bar\Gamma \atop \hbox{ \sevenrm loc 1PI.}\; \bar\Gamma / \bar\gamma \in \Cal T }} \bar\gamma \otimes \bar\Gamma /\bar\gamma.
\end{equation}
\begin{ex}
Taking the same graphs as the example \ref{ex1} we obtain in $\varphi^3$ Theory \cite{A.D2000}: 
\begin{eqnarray*}
\Delta(\grapheexmdel , 0) &=& (\grapheexmdel, 0) \otimes \un  + \un \otimes (\grapheexmdel, 0)\\
&&\\
&+& (\graphecerc, 0)\otimes (\graphecerccroizero, 0) +  (\graphecerc, 1) \otimes (\graphecerccroiun, 0).
\end{eqnarray*}
In $QED$ \cite{wvs06}:
\begin{eqnarray*}
\Delta (\qedj , 1) &=& \un \otimes (\qedj , 1) + (\qedj , 1) \otimes \un \\
&&\\
&+& (\qeda , 0) \otimes (\qedcoiz , 1) + (\qeda , 1) \otimes (\qedcoiu , 1).
\end{eqnarray*}
\end{ex}
\section {External structures}
\subsection{The unordered tensor product}
Let $A$ be a finite set, and let $V_j$ be a vector space for any $j\in A$. The product $\prod_{j\in A} V_{j}$ is defined by:
$$\prod_{j\in A} V_{j} := \{ v : A \longrightarrow \coprod_{j\in A} V_j \;, \; v(i) \; \in \;  V_{i} \;\forall \; i \; \in A \}.$$ 
The space $ V := \bigotimes_{j\in A} V_{j}$ is then defined by the following universal property: for any vector space $E$ and for any multilinear map $F : \prod_{j\in A} V_{j} \longrightarrow E$, there exists a unique linear map $\bar F$ such that the following diagram is commutative:
\diagramme{
\xymatrix{
& \bigotimes_{j\in A} V_{j} \ar[d]^{\bar F} \\
 \prod_{j\in A} V_{j} \ar[ru]^{v \mapsto \bigotimes_{j\in A} v_j} \ar[r]_{F} & E }
 }

\begin{remark}
Let $\big( e_{\lambda} \big)_{\lambda \in {\Lambda_j}} $ be a basis of $V_{j}$. A basis of $\bigotimes_{j\in A} V_{j}$ is given by:
$$ \big( f_\mu = \bigotimes_{j\in A} e_{\mu(j)} \big)_{\mu \in  \Lambda},$$ 
where $ \Lambda = \prod_{j\in A} \Lambda_j =  \{ \mu : A \longrightarrow \coprod_{j\in A} \Lambda_j \;\; \text {such that}\;\;  \mu (j) \in \Lambda_j \}$.  
\end{remark}
\subsection {An algebra of $\Cal C^{\infty}$ functions}
Let $D$ be an integer $\geq 1$ (the dimension). For any half-edge $e$ of $\Gamma$ we denote by $p_e \in  \mathbb R^D$ the corresponding moment. More precisely the moment space of graph $\Gamma$ is defined by:
$$ W_\Gamma = \{ p :\Cal E (\Gamma) \longrightarrow \mathbb R^D ,   \sum_{e \in st(v)} p_e = 0 \;\;\forall v \in \Cal V (\Gamma) ,\;\; p_e + p_{\sigma(e)} = 0\;\; \forall e \in \Cal V (\Gamma) \;\;/ \;\; e \neq\sigma(e) \}.$$
In particular,
$$ W_{\smop{res} \Gamma} = \{ (p_1, \cdots ,p_{\left|\smop {Ext} (\Gamma)\right|} ), p_j  \in \mathbb R^D ,  \sum_{j=1}^{\left|\smop{Ext} (\Gamma)\right|} p_j =0 \}.$$
We introduce then:
$$V_\Gamma = \Cal C^{\infty} ( W_\Gamma ,\mathbb C) \;\; \text {for} \; \Gamma  \; \text{ connected },$$
$$ V= \prod_{\Gamma \smop{connected}}  V_\Gamma,$$
and finally:
\begin{equation}\label{calB}
\Cal B :=  \prod_{\Gamma \smop{connected or not}} V_\Gamma,
\end{equation}
where  the space $ V_\Gamma := \bigotimes_{j\in A} V_{\Gamma_j}$ is the unordered tensor product of the $V_{\Gamma_j}$'s and where the $\Gamma_j$'s are the connected components of $\Gamma$.  The space $ V_{\Gamma}$ is naturally identified with a subspace of $\Cal C^{\infty} ( W_\Gamma ,\mathbb C)$ via: 
$$\bigotimes_{j\in A} v_i (p) := \prod_{j\in A} v_j (p_j) \; \text {with} \; p_j := p_{|\Cal E ({\Gamma_j})}.$$
We equip also $\Cal B$ with the unordered concatenation product denoted by $\bullet$: for $v = \bigotimes_{j\in A} v_j \in V_{\Gamma}$ and $v' = \bigotimes_{j\in B} v_j \in V_{\Gamma'}$ (with $A\cap B=\emptyset$), the product $v \bullet v' \in V_{\Gamma \Gamma'}$  is defined by:
\begin{equation}
v \bullet v' =  \bigotimes_{j\in A \coprod B} v_j.
\end{equation}
The product $\bullet$ is commutative by definition. This definition extends naturally to a bilinear product: $\Cal B \times \Cal B \longrightarrow \Cal B$.
\begin{proposition} \label{concat}
Let $\Gamma_1$ and $\Gamma_2$ be two graphs (not necessarily connected), and let $\Gamma = \Gamma_1\Gamma_2$. For any $v_1 , v'_1 \in  V_{\Gamma_1}$ and $v_2 , v'_2 \in  V_{\Gamma_2}$ we have the following equality in $V_\Gamma$ :
$$ (v_1  v'_1) \bullet (v_2 v'_2) = (v_1 \bullet v_2)  (v'_1 \bullet v'_2).$$
\end{proposition}
\begin{proof}
For $p_1 \in W_{\Gamma_1}$ et $p_2 \in W_{\Gamma_2}$ we have:

\begin{eqnarray*}
(v_1  v'_1) \bullet (v_2 v'_2) ( p_1, p_2) &=& v_1  v'_1( p_1) v_2 v'_2 (p_2)\\
&=& v_1 ( p_1)  v'_1( p_1) v_2 ( p_2) v'_2 (p_2)\\
&=&v_1 ( p_1) v_2 ( p_2) v'_1( p_1) v'_2 (p_2)\\
&=&(v_1 \bullet v_2) ( p_1, p_2)  (v'_1 \bullet v'_2)( p_1, p_2)\\
&=&(v_1 \bullet v_2)(v'_1 \bullet v'_2)( p_1, p_2).
\end{eqnarray*}
\end{proof}
\subsection {Convolution product $\circledast$}
Let $\Gamma$ be a graph and $\gamma$ a covering subgraph of $\Gamma$. We denote by $ i_{\Gamma, \gamma} : V_{\Gamma / \gamma} \hookrightarrow V_{\Gamma }$ and $ \pi_{\gamma, \Gamma} : V_{\gamma} \twoheadrightarrow V_{\Gamma }$ two morphisms of algebras which are defined as follows:\\
Let $\Cal F_{\Gamma,\gamma} : W_\Gamma \rightarrow W_{\Gamma/\gamma}$ the projection of $W_\Gamma $ onto $W_{\Gamma/\gamma}$ by neglecting the internal moments of $\gamma$, that we can still be defined by the following commutative diagram:
\diagramme{
\xymatrix{
& \Cal E (\Gamma) \ar[d]^{p} \\
 \Cal E (\Gamma/\gamma) \ar@{^{(}->}[ru]^{\mop{inj}(\Gamma, \gamma)} \ar[r]_{\Cal F_{\Gamma,\gamma}(p)} & \mathbb R^D}
} 
where $\mop{inj}(\Gamma, \gamma)$ is the natural injection and $\Cal F_{\Gamma,\gamma} = \mop{inj}(\Gamma, \gamma)^\ast$.
We now consider the following commutative diagram:
\diagramme{
\xymatrix{
& W_{\Gamma/\gamma} \ar[d]^{f} \\
 W_{\Gamma} \ar@{->>}[ru]^{\Cal F_{\Gamma,\gamma}} \ar[r]_{{\Cal F_{\Gamma,\gamma}}^\ast f} & \mathbb C}
} 
\noindent We define the injection $i_{\Gamma, \gamma}$ by : $i_{\Gamma, \gamma} = {\Cal F_{\Gamma,\gamma}}^\ast$.\\
We denote by $\Cal G_{\gamma,\Gamma}:  W_\Gamma \hookrightarrow W_\gamma$ the natural inclusion of $W_\Gamma$ in $W_\gamma$ and we consider the following commutative diagram:

\diagramme{
\xymatrix{
& W_{\gamma} \ar[d]^{f} \\
 W_{\Gamma} \ar@{^{(}->}[ru]^{\Cal G_{\gamma,\Gamma}} \ar[r]_{{\Cal G_{\gamma,\Gamma}}^\ast f} & \mathbb C}
} 
\noindent We define the surjection $\pi_{\gamma, \Gamma}$ by:  $\pi_{\gamma, \Gamma}: ={\Cal G_{\gamma,\Gamma}}^\ast :  V_{\gamma} \twoheadrightarrow V_{\Gamma }$.

Let $ \wt{\Cal H}_{\Cal T}$ be the specified Feynman graphs bialgebra associated with a theory $\Cal T$. We denote by ${\Cal L} ( \wt{\Cal H}_{\Cal T}, \Cal B)$ the space of $\mathbb C $-linear maps  $\chi :\wt{\Cal H}_{\Cal T} \longrightarrow \Cal B$, and by $\wt {\Cal L} ( \wt{\Cal H}_{\Cal T}, \Cal B)$ the subspace of ${\Cal L} ( \wt{\Cal H}_{\Cal T}, \Cal B)$ of $\chi$ such that:
\begin{enumerate}
\item $\chi$ does not depend on the specification of $\Gamma$, in other words: $\chi( \Gamma ,\underline{i}) = \chi( \Gamma)$.
\item $ \chi (\Gamma) \in V_\Gamma$  for any graph $\Gamma$, i.e. the projection of $\chi(\Gamma)$ on $V_{\Gamma'}$ vanishes for any graph $\Gamma'\not=\Gamma$.
\item $ \chi (\Gamma) = \un_{V_\Gamma}$ if $\Gamma$ has no internal edges, where $\un_{V_\Gamma}$ denotes the constant function equal to $\un$ on $W_\Gamma$. 
\end{enumerate} 
Then we define a convolution product $\circledast$ for all $\chi$, $\eta \in \wt {\Cal L} ( \wt{\Cal H}_{\Cal T}, \Cal B)$ and for all specified graphs  $(\Gamma, \underline{i})$ by:
\begin{equation}\label{conv}
(\chi \circledast \eta) (\Gamma , \underline{i}) = (\chi \circledast \eta) (\Gamma) = \sum_{(\gamma,\underline j)\subset(\Gamma, \underline i) \atop (\Gamma, \underline i) / (\gamma,\underline j) \in \Cal T } \pi_{\gamma, \Gamma} [\chi (\gamma )] i_{\Gamma, \gamma} \big[ \eta \big(\Gamma/(\gamma,\underline j)\big) \big].
\end{equation}
The product used in the right hand side is the pointwise product in $V_\Gamma$.
\begin{theorem}
The product $\circledast$ is associative.
\end{theorem}
\begin{proof}
Let $ \chi , \eta$ and $\xi$ be three elements  of $\wt {\Cal L} ( \wt{\Cal H}_{\Cal T}, \Cal B)$ and $(\Gamma , \underline{i})$ a specified graph. We denote indifferently $\bar\Gamma = (\Gamma,\underline{i})$, $\bar\gamma = (\gamma,\underline{j})$, $\bar\delta = (\delta,\underline{k})$ and $\bar\Gamma/\bar\gamma = (\Gamma/(\gamma,\underline j), \underline{i})$. First, we have:
\begin{eqnarray*}
\chi \circledast (\eta \circledast \xi) (\Gamma , \underline{i})&=&\chi \circledast (\eta \circledast \xi) (\Gamma)\\
&=& \sum_{\bar\delta \subset \; \bar\Gamma \atop \bar\Gamma / \bar\delta \in \Cal T} \pi_{\delta, \Gamma} [\chi (\delta)] \;i_{\Gamma, \delta} \Big[ (\eta \circledast \xi)(\bar\Gamma/\bar\delta) \Big]\\
&=& \sum_{\bar\delta \subset \; \bar\Gamma \atop \bar\Gamma / \bar\delta \in \Cal T} \pi_{\delta, \Gamma} [\chi (\delta)] \;i_{\Gamma, \delta} \Big[ \sum_{\bar\alpha \subset \; \bar\Gamma/\delta \atop (\bar\Gamma / \bar\delta)/ \bar\alpha  \in \Cal T} \pi_{\alpha, \Gamma/\delta} [\eta (\alpha)] i_{\Gamma/\delta, \alpha} [ \xi((\bar\Gamma/\bar\delta )/\bar\alpha) ] \Big].
\end{eqnarray*}
By identifying $\bar\alpha$ with $\bar\gamma/\bar\delta$ where $\bar\gamma$ is a subgraph of $\bar\Gamma$ containing $\bar\delta$, and $(\bar\Gamma /\bar \delta )/ \bar\alpha$ with $\bar\Gamma / \bar\gamma$ we obtain: 
\begin{eqnarray*}
\chi \circledast (\eta \circledast \xi) (\Gamma , \underline{i})&=&\chi \circledast (\eta \circledast \xi) (\Gamma)\\
&=& \sum_{\bar\delta \subset \; \bar\gamma \subset \; \bar\Gamma \atop \bar\gamma / \bar\delta \;;\; \bar\Gamma / \bar\gamma \in \Cal T} \pi_{\delta, \Gamma} [\chi (\delta)] \;i_{\Gamma, \delta} \Big[ \pi_{\gamma/\delta, \Gamma/\delta} [\eta (\bar\gamma/\bar\delta)] i_{\Gamma/\delta, \gamma/\delta} [ \xi(\bar\Gamma/\bar\gamma )] \Big]\\
&=& \sum_{\bar\delta \subset \; \bar\gamma \subset \; \bar\Gamma \atop \bar\gamma / \bar\delta \;;\; \bar\Gamma / \bar\gamma \in \Cal T} \pi_{\delta, \Gamma} [\chi (\delta)] \;i_{\Gamma, \delta} \pi_{\gamma/\delta, \Gamma/\delta} [\eta (\bar\gamma/\bar\delta)] i_{\Gamma, \delta} i_{\Gamma/\delta, \gamma/\delta} [ \xi(\bar\Gamma/\bar\gamma )].
\end{eqnarray*}
Secondly we have:
\begin{eqnarray*}
(\chi \circledast \eta) \circledast \xi (\Gamma , \underline{i})&=&(\chi \circledast \eta) \circledast \xi (\Gamma)\\
&=&\sum_{\bar\gamma \subset \; \bar\Gamma  \atop \bar\Gamma / \bar\gamma \in \Cal T} \pi_{\gamma, \Gamma} [ \chi \circledast \xi(\gamma)] \;i_{\Gamma, \gamma} [ \xi(\bar\Gamma/\bar\gamma)]\\
&=& \sum_{\bar\delta \subset \; \bar\gamma \subset \; \bar\Gamma \atop \bar\gamma / \bar\delta \;;\; \bar\Gamma / \bar\gamma \in \Cal T} \pi_{\gamma, \Gamma} \Big[ \pi_{\delta, \gamma} \chi (\delta) i_{\delta, \gamma} [ \eta(\bar\gamma/\bar\delta)] \Big] \;i_{\Gamma, \gamma} [ \xi(\bar\Gamma/\bar\gamma)]\\
&=& \sum_{\bar\delta \subset \; \bar\gamma \subset \; \bar\Gamma \atop \bar\gamma / \bar\delta \;;\; \bar\Gamma / \bar\gamma \in \Cal T} \pi_{\gamma, \Gamma}  \pi_{\delta, \gamma} [\chi (\delta)] \pi_{\gamma, \Gamma} i_{\delta, \gamma} [ \eta(\bar\gamma/\bar\delta)] \;i_{\Gamma, \gamma} [ \xi(\bar\Gamma/\bar\gamma)].
\end{eqnarray*}
The two following diagrams commute: 
\diagramme{
\xymatrix{
W_{\Gamma} \ar[d]_{\Cal G_{\gamma,\Gamma}} \ar[r]^{\Cal F_{\Gamma,\delta}}
	&W_{\Gamma/\delta} \ar[d]^{\Cal G_{\gamma/\delta,\Gamma/\delta}}	\\
W_\gamma \ar[r]_{\Cal F_{\gamma,\delta}}& W_{\gamma/\delta}}
\hskip 20mm
\xymatrix{
& W_{\Gamma/\delta} \ar[d]^{\Cal F_{\Gamma/\delta,\gamma/\delta}} \\
 W_{\Gamma} \ar[ru]^{\Cal F_{\Gamma,\delta}} \ar[r]_{\Cal F_{\Gamma,\gamma}} & W_{\Gamma/\gamma} }
 }
From the two preceding diagrams we obtain the following two commutative diagrams:
\diagramme{
\xymatrix{
V_{\gamma/\delta} \ar[d]_{\pi_{\gamma/\delta , \Gamma/\delta}} \ar[r]^{i_{\gamma,\delta}}
	&V_\gamma \ar[d]^{\pi_{\gamma , \Gamma}}	\\
V_{\Gamma/\delta} \ar[r]_{i_{\Gamma,\delta}}& V_{\Gamma}}
\hskip 20mm
\xymatrix{
& V_{\Gamma/\delta} \ar[d]^{\pi_{\Gamma ,\delta}} \\
V_{\Gamma/\gamma} \ar[ru]^{i_{\Gamma/\delta ,\gamma/\delta}} \ar[r]_{i_{\Gamma ,\gamma}} & V_{\Gamma}}
} 
Hence we can write:
\begin{eqnarray*}
(\chi \circledast \eta )\circledast \xi (\Gamma , \underline{i}) &=& \sum_{\bar\delta \subset \; \bar\gamma \subset \; \bar\Gamma \atop \bar\gamma / \bar\delta \;;\; \bar\Gamma / \bar\gamma \in \Cal T} \pi_{\delta, \Gamma} [\chi (\delta)] \;i_{\Gamma, \delta} \pi_{\gamma/\delta, \Gamma/\delta} [\eta (\bar\gamma/\bar\delta)] i_{\Gamma, \delta} i_{\Gamma/\delta, \gamma/\delta} [ \xi(\bar\Gamma/\bar\gamma )].
\end{eqnarray*}
Consequently, as $(\chi \circledast \eta )\circledast \xi (\Gamma , \underline{i})  = \chi \circledast (\eta \circledast \xi) (\Gamma , \underline{i})$ for any specified graph $(\Gamma , \underline{i})$, the product $\circledast$ is associative.
\end{proof}
\begin{theorem}
Let $G = \{ \varphi \in \wt {\Cal L} ( \wt{\Cal H}_{\Cal T}, \Cal B) \;\;\text{such that}\;\; \varphi (\gamma \gamma') = \varphi (\gamma) \bullet \varphi (\gamma') \text{ and } \varphi (\un) = \un_\Cal B \}$. Equipped with the product $\circledast$, the set $G$ is a subgroup of the semigroup of characters of $\wt{\Cal H}_{\Cal T}$ with values in $\Cal B$. 
\end{theorem}
\begin{proof}
Let $\varphi$, $\psi$ be two elements of $G$ and $\bar\Gamma = (\Gamma , i)$, $\bar\Gamma' = (\Gamma' , i')$ two specified graphs:
It is clear that by definition:  $\varphi \circledast \psi \in \wt {\Cal L} ( \wt{\Cal H}_{\Cal T}, \Cal B)$. Using Proposition \ref{concat} we have then: 
\begin{eqnarray*}
(\varphi \circledast \psi) (\bar\Gamma \bar\Gamma') &=&(\varphi \circledast \psi) (\Gamma \Gamma')\\
&=& \sum_{\bar\gamma \bar\gamma' \subset \; \bar\Gamma \bar{\Gamma'} \atop \bar\Gamma \bar\Gamma' / \bar\gamma\bar\gamma' \in \Cal T} \pi_{\gamma\gamma', \Gamma\Gamma'} [ \varphi (\gamma \gamma')] \; i_{\Gamma\Gamma', \gamma\gamma'} \Big[ \psi ( \bar\Gamma \bar\Gamma'/ \bar\gamma \bar\gamma')\Big]\\  
&=& \sum_{\bar\gamma  \subset\; \bar\Gamma \;,\;\bar\gamma'  \subset\; \bar\Gamma' \atop \bar\Gamma / \bar\gamma \;;\; \bar\Gamma' / \bar\gamma' \in \Cal T} \Big ( \pi_{\gamma, \Gamma}[ \varphi (\gamma)] \bullet \pi_{\gamma', \Gamma'}[ \varphi (\gamma')] \Big) \; \Big ( i_{\Gamma, \gamma} \big[ \psi (\bar\Gamma/ \bar\gamma)\big] \bullet i_{\Gamma', \gamma'} \Big[\psi (\bar\Gamma'/ \bar\gamma') \Big] \Big)\\ 
&=& \sum_{\bar\gamma  \subset\; \bar\Gamma \;,\;\bar\gamma'  \subset\; \bar\Gamma' \atop \bar\Gamma / \bar\gamma \;;\; \bar\Gamma' / \bar\gamma' \in \Cal T} \Big( \pi_{\gamma, \Gamma}[ \varphi (\gamma)] i_{\Gamma, \gamma} \big[ \psi (\bar\Gamma/ \bar\gamma)\big] \Big) \bullet \Big( \pi_{\gamma', \Gamma'}[ \varphi (\gamma')]\; i_{\Gamma', \gamma'} \big[\psi (\bar\Gamma'/ \bar\gamma') \big] \Big)\\
&=&(\varphi \circledast \psi) (\bar\Gamma) (\varphi \circledast \psi) (\bar\Gamma').
\end{eqnarray*}
The identity element $e$ is defined by: 
\begin{equation}
e(\bar\Gamma) = \left\lbrace
\begin{array}{lcl}
\un_{V_\Gamma}  \;\;\;\;\; \text{if} \;\; \bar\Gamma \;\; \text{is a specified graph of degree zero}\\
0  \;\; \;\; \;\; \;\; \text{if not}.
\end{array}\right. 
\end{equation}
Indeed for any $\varphi \in \wt {\Cal L} ( \wt{\Cal H}_{\Cal T}, \Cal B)$ we have:\\
If $\bar\Gamma$ of degree zero,
$$ (e \circledast \varphi)(\bar\Gamma) =e(\bar\Gamma)\varphi(\bar\Gamma) = \varphi(\Gamma)$$
Similarly:   $$ (\varphi \circledast e)(\bar\Gamma) =\varphi(\bar\Gamma)e(\bar\Gamma) = \varphi(\Gamma)$$
If $\bar\Gamma$ of degree $\geq 1$ we have:
\begin{eqnarray*}
(e \circledast \varphi)(\bar\Gamma) &=& \sum_{(\bar\Gamma)} \pi_{\gamma, \Gamma} [ e (\bar\gamma )] \; i_{\Gamma, \gamma} [\varphi (\bar\Gamma/ \bar\gamma)]\\
&=&\pi_{\smop{sk}\Gamma, \Gamma} [ \un_{V_{\smop{sk}\Gamma}}] \; i_{\Gamma, \smop{sk}\Gamma} [\varphi ( \Gamma )] \\
&=& \varphi (\Gamma ).
\end{eqnarray*}
\begin{eqnarray*}
(\varphi \circledast e)(\bar\Gamma) &=& \sum_{(\Gamma)} \pi_{\gamma, \Gamma} [ \varphi (\gamma )] \; i_{\Gamma, \gamma} [e (\bar\Gamma/ \bar\gamma)]\\
&=&\pi_{\Gamma, \Gamma} [\varphi (\Gamma )] \; i_{\Gamma, \Gamma} [e ( \mop{sk}\bar\Gamma )] \\
&=& \varphi (\Gamma ).
\end{eqnarray*}
The inverse of an element $\varphi$ of $G$ is given by the following formula:
\begin{eqnarray*}
\varphi^{\circledast -1} (\bar\Gamma) &=&( e - (e-\varphi))^{\circledast -1} (\bar\Gamma)\\ 
&=&\sum_{n} (e-\varphi)^{\circledast n} (\bar\Gamma).
\end{eqnarray*}
This sum is well defined: it stops at $n=q$ for specified graph $\bar\Gamma$ of degree $q$. Then we have: 
$$\varphi^{\circledast-1} \circledast \varphi = \varphi\circledast \varphi^{\circledast -1} = e.$$
\end{proof}
\subsection {Birkhoff decomposition} 
In this section we will explain how to renormalize a character $ \varphi$ of the specified graphs graded bigebra $\wt{\Cal H}_{\Cal T}$: Let $\varphi$ be a character  with values in the unitary commutative algebra ${\Cal A} : = \Cal B [z^{-1} , z]] $ equipped with the minimal subtraction scheme:
\begin{equation}
{\Cal A} = {\Cal A}_- \oplus {\Cal A}_+
\end{equation}
where: $$ {\Cal A}_+ : =   \Cal B [[ z ]],$$
       $$ {\Cal A}_- : =  z^{-1} \Cal B [z^{-1}].$$  

${\Cal A}_-$ and ${\Cal A}_+$ are two subalgebras of ${\Cal A}$, with  $\textbf{1}_{{\Cal A}} \in {\Cal A}_+$. We denote by $P$ the projection on ${\Cal A_-}$ parallel to ${\Cal A}_+$. The space of linear maps of $\wt{\Cal H}_{\Cal T}$ to ${\Cal A}$ is equipped with the convolution product $\circledast$ defined by the formula \eqref{conv}.
We have verified in the previous paragraph that the space of characters $\wt{\Cal H}_{\Cal T}$ with values in ${\Cal A}$ is a group for the convolution product $\circledast$.
\begin{theorem}
\begin{enumerate}
\item Any character $ \varphi \in G$ has a unique Birkhoff decomposition in $G$ : 
\begin{equation}
\varphi = \varphi _-^{\circledast-1} \circledast \varphi_+
\end{equation}
compatible with the renormalization scheme chosen, in other words, such that $\varphi_+$ takes its values in ${\Cal A}_+ $ and such that $ \varphi_- (\bar\Gamma) \in {\Cal A}_- $ for any specified graph $(\Gamma , \underline{i}) $ of degree $\geq 1$. The components $\varphi_+$ and $\varphi_- $ are given by simple recursive formulas: for any $\bar\Gamma$ of degree zero (i.e without internal edges) we put: $\varphi _-(\bar\Gamma) = \varphi _+(\bar\Gamma) = \varphi (\bar\Gamma)= \un_{V_\Gamma}$. If we assume that $ \varphi_- (\bar\Gamma) $ and $\varphi_+ (\bar\Gamma) $ are known for $ \bar\Gamma $ of degree $ k \leq n-1 $, we have then for any specified graph $\bar\Gamma $ of degree $n$:
\begin{equation}
\varphi_- (\bar\Gamma) = \varphi_- (\Gamma) = - P \Big( \varphi (\Gamma) + \sum_{\bar\gamma \subsetneq\; \bar\Gamma\atop \bar\Gamma / \bar\gamma \in \Cal T} \pi_{\gamma, \Gamma} [\varphi_- (\gamma)] \;i_{\Gamma, \gamma} \Big[ \varphi \left( \Gamma / (\gamma ,\underline{j})\right) \Big ]\Big) \label{bbbb}
\end{equation}
\begin{equation}
\varphi_+ (\bar\Gamma) =\varphi_+ (\Gamma) = (I- P) \Big( \varphi (\Gamma) + \sum_{\bar\gamma \subsetneq\; \bar\Gamma \atop \bar\Gamma / \bar\gamma \in \Cal T} \pi_{\gamma, \Gamma} [\varphi_- (\gamma)] \;i_{\Gamma, \gamma} \Big[ \varphi \left( \Gamma / (\gamma ,\underline{j})\right) \Big ]\Big). \label{bbbbbbbbb1}
\end{equation}
\item $\varphi_+$ and $\varphi_-$ are two characters. We will call $\varphi_+$ the renormalized character and $\varphi_-$ the character of the counterterms.
\end{enumerate}
\end{theorem}
\begin{proof}
\begin{enumerate}
\item The fact that $\varphi_+$ takes its values in ${\Cal A}_+$ and that $\varphi_- (\Gamma) \in {\Cal A}_-$ is immediate by definition of $P$, and we can verify by a simple calculation that $\varphi_+ = \varphi _- \circledast \varphi$:
\begin{eqnarray*}
\varphi_+ (\Gamma) &=& (I- P) \Big( \varphi (\Gamma) + \sum_{\bar\gamma \subsetneq\; \bar\Gamma \atop \bar\Gamma / \bar\gamma \in \Cal T} \pi_{\gamma, \Gamma} [\varphi_- (\gamma)] \;i_{\Gamma, \gamma} \Big[ \varphi \left( \Gamma / (\gamma ,\underline{j})\right) \Big ]\Big) \\
&=& \varphi (\Gamma) + \varphi_- (\Gamma) + \sum_{\bar\gamma \subsetneq\; \bar\Gamma \atop \bar\Gamma / \bar\gamma \in \Cal T} \pi_{\gamma, \Gamma} [\varphi_- (\gamma)] \;i_{\Gamma, \gamma} \Big[ \varphi \left( \Gamma / (\gamma ,\underline{j})\right) \Big ].
\end{eqnarray*}
By using the fact that $\varphi _-(\Gamma) = \varphi (\Gamma)= \un_{V_\Gamma}$, for any graph $\Gamma$ of degree zero we have: 
\begin{eqnarray*}
\varphi_- \circledast \varphi  (\Gamma) &=&  \sum_{\bar\gamma \subset\; \bar\Gamma \atop \bar\Gamma / \bar\gamma \in \Cal T} \pi_{\gamma, \Gamma} [\varphi_- (\gamma)] \;i_{\Gamma, \gamma} \Big[ \varphi \left( \Gamma / (\gamma ,\underline{j})\right) \Big ]\\
&=& \pi_{\smop{sk}\Gamma,\Gamma} [\varphi_- (\mop{sk}\Gamma)] i_{\Gamma, \smop{sk}\Gamma} [\varphi (\Gamma)]  + \pi_{\Gamma, \Gamma} [ \varphi_- (\Gamma)] i_{\Gamma, \Gamma} [\varphi ( \mop{res}\Gamma)]\\
&+& \sum_{\bar\gamma \subsetneq\; \bar\Gamma \atop \bar\Gamma / \bar\gamma \in \Cal T} \pi_{\gamma, \Gamma} [\varphi_- (\gamma)] \; i_{\Gamma, \gamma} \Big[ \varphi \left( \Gamma / (\gamma ,\underline{j})\right) \Big ]\\
&=& \varphi (\Gamma) + \varphi_- (\Gamma) + \sum_{\bar\gamma \subsetneq\; \bar\Gamma \atop \bar\Gamma / \bar\gamma \in \Cal T} \pi_{\gamma, \Gamma} [\varphi_- (\gamma)] \; i_{\Gamma, \gamma} \Big[ \varphi \left( \Gamma / (\gamma ,\underline{j})\right) \Big ].
\end{eqnarray*}
Hence: $\varphi_+ = \varphi_- \circledast \varphi$ is equivalent to saying that: $\varphi = \varphi _-^{\circledast-1} \circledast \varphi_+$.\\
We now assume that $\varphi = \varphi _-^{\circledast-1} \circledast \varphi_+ = \psi _-^{\circledast-1} \circledast \psi_+$. Thus we obtain:
$$  \varphi _+ \circledast \psi_+^{\circledast-1} = \varphi _- \circledast \psi_-^{\circledast-1}.$$
The right-hand side of the equality sends any specified graph of degree $\geq1$ in ${\Cal A}_+$ but the left-hand side sends in ${\Cal A}_-$, then for any graph $\bar\Gamma$ of degree $\geq1$ we have:
   $$  \varphi _+ \circledast \psi_+^{\circledast-1}(\Gamma) = \varphi _- \circledast \psi_-^{\circledast-1}(\Gamma) = 0.$$
Then: $\varphi _+ \circledast \psi_+^{\circledast-1} = \varphi _- \circledast \psi_-^{\circledast-1} = e$, which proves the uniqueness of the Birkhoff decomposition.
\item We will just prove that $\varphi_-$ is a character. Then  $\varphi_+ = \varphi_- \circledast \varphi$ is also a character. The idea follows from the fact that the projection $P$ satisfies the Rota-Baxter equality:  
\begin{equation}\label{Rb}
P (a)P (b) = P \Big( -ab +  P (a)b + P (b) a \Big).
\end{equation}
Let $\varphi$ be an element of $G$. The proof is obtained by induction on the degree of the graph $\Gamma\Gamma'$. For $\bar\Gamma\bar\Gamma'$ of degree zero we have $\un_{V_\Gamma}\bullet\un_{V_\Gamma'}=\un_{V_{\Gamma\Gamma'}}$. We assume that $\varphi_- (\Gamma \Gamma') = \varphi_- (\Gamma)\bullet\varphi_- (\Gamma')$ for any $ \bar\Gamma$,$\bar\Gamma'  \in  \wt{\Cal H}_{\Cal T}$ such that: $|\bar\Gamma| + |\bar\Gamma'| \leq d-1$ and show the equality for $\bar\Gamma$, $\bar\Gamma'  \in  \wt{\Cal H}_{\Cal T}$ such that: $|\bar\Gamma| + |\bar\Gamma'| = d$, where $|\bar\Gamma|$ denotes the degree of $\bar\Gamma$.\\
We have :
$$\varphi_- (\Gamma)\bullet\varphi_- (\Gamma') = P(X)\bullet P(Y),$$
where:
$$ X = \varphi (\Gamma) + \sum_{\bar\gamma \subsetneq\; \bar\Gamma \atop \bar\Gamma / \bar\gamma \in \Cal T} \pi_{\gamma, \Gamma} [\varphi_- (\gamma)] \; i_{\Gamma, \gamma} \Big[ \varphi \left( \Gamma / (\gamma ,\underline{j})\right) \Big ]$$ 
$$Y = \varphi (\Gamma') + \sum_{\bar\gamma' \subsetneq\; \bar\Gamma' \atop \bar\Gamma' / \bar\gamma' \in \Cal T} \pi_{\gamma', \Gamma'} [\varphi_- (\gamma')] \; i_{\Gamma', \gamma'} \Big[ \varphi \left( \Gamma' / (\gamma' ,\underline{j})\right) \Big ].$$
We have: $$  \varphi_- (\Gamma)\bullet\varphi_- (\Gamma') = P(X)\bullet P(Y) = P \Big( -X \bullet Y +  P (X)\bullet Y +X \bullet P (Y) \Big).$$
Since  $P(X) = - \varphi_- (\Gamma)$ and $P(Y) = - \varphi_- (\Gamma')$, we obtain:
$$  \varphi_- (\Gamma)\bullet\varphi_- (\Gamma') = -P \Big(  X \bullet Y + \varphi_- (\Gamma)\bullet Y + X \bullet \varphi_- (\Gamma') \Big).$$
Therefore: 
\begin{eqnarray*}
\varphi_- (\Gamma)\bullet\varphi_- (\Gamma') &=&  - P \Big[\varphi (\Gamma)\bullet \varphi (\Gamma') + \varphi_- (\Gamma) \bullet\varphi (\Gamma') +\varphi (\Gamma) \bullet\varphi_- (\Gamma')\\
&+& \sum_{\bar\gamma \subsetneq\; \bar\Gamma \atop \bar\Gamma / \bar\gamma \in \Cal T} \Big(\pi_{\gamma, \Gamma}  [ \varphi_- ( \gamma )] \; i_{\Gamma, \gamma} \big[\varphi \left(\bar\Gamma / \bar\gamma \right)\big]\Big) \bullet \Big(\varphi_- (\Gamma') + \varphi (\Gamma') \Big) \\
&+& \sum_{\bar\gamma' \subsetneq\; \bar\Gamma' \atop \bar\Gamma' / \bar\gamma' \in \Cal T} \Big( \pi_{\gamma', \Gamma'} [ \varphi_- ( \gamma' )] \; i_{\Gamma', \gamma'} \big[\varphi ( \bar\Gamma' / \bar\gamma' )\big]\Big) \bullet\Big( \varphi_- (\Gamma) + \varphi (\Gamma)\Big) \\
&+&\sum_{\bar\gamma \subsetneq\; \bar\Gamma \;,\; \bar\gamma' \subsetneq\; \bar\Gamma' \atop \bar\Gamma / \bar\gamma \;;\; \bar\Gamma' / \bar\gamma' \in \Cal T} \Big(\pi_{\gamma, \Gamma}  [ \varphi_- ( \gamma )] \;i_{\Gamma, \gamma} \big[ \varphi (\bar\Gamma / \bar\gamma) \big]\Big)\bullet\Big( \pi_{\gamma', \Gamma'} [ \varphi_- ( \gamma' )] \;  i_{\Gamma', \gamma'} \big[\varphi (\bar\Gamma' / \bar\gamma') \big]\Big) \Big].
\end{eqnarray*}
The coproduct $\Delta(\bar\Gamma \bar\Gamma')$ is given by:
\begin{eqnarray*}
\Delta (\bar\Gamma \bar\Gamma') &=&  \bar\Gamma \bar\Gamma' \otimes \mop{res}\bar\Gamma\mop{res}\bar\Gamma' +  \mop{sk}\bar\Gamma\mop{sk}\bar\Gamma' \otimes \bar\Gamma \bar\Gamma' + \bar\Gamma\mop{sk}\bar\Gamma'\otimes \bar\Gamma'\mop{res}\bar\Gamma + \bar\Gamma'\mop{sk}\bar\Gamma\otimes \bar\Gamma\mop{res}\bar\Gamma'\\
&+& \sum_{\bar\gamma \subsetneq\; \bar\Gamma \atop \bar\Gamma / \bar\gamma \in \Cal T } \bar\gamma\bar\Gamma' \otimes (\bar\Gamma / \bar\gamma) \mop{res}\bar\Gamma' + \bar\gamma \mop{sk}\bar\Gamma'\otimes (\bar\Gamma / \bar\gamma) \bar\Gamma' \\
&+&\sum_{\bar\gamma' \subsetneq\; \bar\Gamma' \atop \bar\Gamma' / \bar\gamma' \in \Cal T} \bar\Gamma\bar\gamma' \otimes (\bar\Gamma' / \bar\gamma') \mop{res}\bar\Gamma + \bar\gamma'\mop{sk}\bar\Gamma \otimes\bar\Gamma (\bar\Gamma' / \bar\gamma') + \sum_{\bar\gamma \subsetneq \bar\Gamma\;;\;\bar\gamma' \subsetneq\; \bar\Gamma' \atop \bar\Gamma / \bar\gamma \;;\; \bar\Gamma' / \bar\gamma' \in \Cal T} \bar\gamma\bar\gamma' \otimes (\bar\Gamma/\bar\gamma) (\bar\Gamma'/\bar\gamma').
\end{eqnarray*}
Since \;$ \varphi_-(\Gamma \Gamma') = -P \Big( \varphi_- \circledast \varphi (\Gamma \Gamma') - \varphi_-(\Gamma \Gamma') \Big)$, we have: 
\begin{eqnarray*}
\varphi_- (\Gamma \Gamma') &=& - P \Big[ \pi_{\Gamma\Gamma', \Gamma\Gamma'}[ \varphi_-(\Gamma\Gamma')] i_{\Gamma\Gamma', \Gamma\Gamma'} [\varphi (\mop{res}\bar\Gamma\mop{res}\bar\Gamma')] \\
&+& \pi_{\smop{sk}\Gamma\smop{sk}\Gamma',\Gamma\Gamma'} [ \varphi_- (\mop{sk}\bar\Gamma\mop{sk}\bar\Gamma')] i_{\Gamma\Gamma',\smop{sk}\Gamma\smop{sk}\Gamma'} [\varphi (\Gamma\Gamma')] \\
&+& \pi_{\Gamma\smop{sk}\Gamma',\Gamma\Gamma'} [ \varphi_- (\Gamma\mop{sk}\bar\Gamma' )] i_{\Gamma\Gamma',\Gamma\smop{sk}\Gamma'}\big[ \varphi( \Gamma'\mop{res}\bar\Gamma)\big]\\
&+&\pi_{\Gamma'\smop{sk}\Gamma,\Gamma\Gamma'}[ \varphi_- (\Gamma'\mop{sk}\bar\Gamma )] i_{\Gamma\Gamma',\Gamma'\smop{sk}\Gamma}[\varphi(\Gamma \mop{res}\bar\Gamma' )] \\
&+& \hskip-5mm\sum_{\bar\gamma \subsetneq\; \bar\Gamma \atop \bar\Gamma / \bar\gamma \in \Cal T}  \pi_{\gamma\gamma',\Gamma\Gamma'} [\varphi_- (\gamma \Gamma')] i_{\Gamma\Gamma',\gamma\gamma'}\big[ \varphi(\bar\Gamma / \bar\gamma\mop{res}\bar\Gamma' )\big] +  \pi_{\gamma\gamma',\Gamma\Gamma'} [\varphi_- (\gamma \mop{sk}\bar\Gamma')] i_{\Gamma\Gamma',\gamma\gamma'}\big[ \varphi(\bar\Gamma / \bar\gamma\Gamma' )\big]\\
&+&\hskip-5mm\sum_{\bar\gamma' \subsetneq\; \bar\Gamma' \atop \bar\Gamma' / \bar\gamma' \in \Cal T} \hskip-2mm\pi_{\gamma\gamma',\Gamma\Gamma'} [\varphi_- (\gamma' \Gamma)] i_{\Gamma\Gamma',\gamma\gamma'}\big[ \varphi(\bar\Gamma' / \bar\gamma'\mop{res}\bar\Gamma )\big] +  \pi_{\gamma\gamma',\Gamma\Gamma'} [\varphi_- (\gamma' \mop{sk}\bar\Gamma)] i_{\Gamma\Gamma',\gamma\gamma'}\big[ \varphi(\bar\Gamma'/\bar\gamma' \Gamma )\big]\\
&+& \sum_{\bar\gamma \subsetneq \bar\Gamma\;;\;\bar\gamma' \subsetneq\; \bar\Gamma' \atop \bar\Gamma / \bar\gamma \;;\; \bar\Gamma' / \bar\gamma' \in \Cal T} \pi_{\gamma\gamma',\Gamma\Gamma'} [\varphi_-(\gamma\gamma')] i_{\Gamma\Gamma',\gamma\gamma'}\big[ \varphi(\bar\Gamma/\bar\gamma \bar\Gamma'/\bar\gamma')\big]  - \varphi_-(\Gamma\Gamma') \Big].
\end{eqnarray*}
We notice that the first and last terms in the right side cancel each other. Since $\varphi$ is a character, $\varphi(\mop{sk}\bar\Gamma)= \varphi_-(\mop{sk}\bar\Gamma) = \un_{V_\Gamma}$ and by the induction hypothesis we obtain:
\begin{eqnarray*}
\varphi_- (\Gamma\Gamma') &=& -P \Big[ \varphi (\Gamma)\bullet \varphi (\Gamma')\\
&+&\Big( \pi_{\Gamma\smop{sk}\Gamma',\Gamma\Gamma'} [ \varphi_- (\Gamma)]\bullet \pi_{\Gamma\smop{sk}\Gamma',\Gamma\Gamma'} [ \varphi_- (\mop{sk}\bar\Gamma' )]\Big)  \Big( i_{\Gamma\Gamma',\Gamma\smop{sk}\Gamma'}\big[ \varphi( \Gamma')\big] \bullet i_{\Gamma\Gamma',\Gamma\smop{sk}\Gamma'}\big[ \varphi( \mop{res}\bar\Gamma)\big] \Big)\\
&+&\Big( \pi_{\Gamma'\smop{sk}\Gamma,\Gamma\Gamma'}[ \varphi_- (\Gamma')]\bullet \pi_{\Gamma'\smop{sk}\Gamma,\Gamma\Gamma'}[ \varphi_- (\mop{sk}\bar\Gamma )] \Big)  \Big(i_{\Gamma\Gamma',\Gamma'\smop{sk}\Gamma}[\varphi(\Gamma )] \bullet i_{\Gamma\Gamma',\Gamma'\smop{sk}\Gamma}[\varphi(\mop{res}\bar\Gamma' )]  \Big) \\
&+& \sum_{\bar\gamma \subsetneq\; \bar\Gamma\atop \bar\Gamma / \bar\gamma \in \Cal T} \Big( \pi_{\gamma,\Gamma} [\varphi_- (\gamma)] i_{\Gamma,\gamma}\big[ \varphi(\bar\Gamma/\bar\gamma)\big]\Big)\bullet \Big( \varphi_- (\Gamma') + \varphi (\Gamma')\Big) \\
&+&\sum_{\bar\gamma' \subsetneq\; \bar\Gamma'\atop \bar\Gamma' / \bar\gamma' \in \Cal T} \Big(\pi_{\gamma',\Gamma'} [ \varphi_- (\gamma' )] i_{\Gamma',\gamma'}\big[ \varphi( \bar\Gamma'/\bar\gamma')\big]\Big) \bullet\Big( \varphi_- (\Gamma) + \varphi (\Gamma) \Big)\\
&+&\sum_{\bar\gamma \subsetneq \bar\Gamma\;;\;\bar\gamma' \subsetneq\; \bar\Gamma' \atop \bar\Gamma / \bar\gamma \;;\; \bar\Gamma' / \bar\gamma' \in \Cal T} \Big(\pi_{\gamma, \Gamma}  [ \varphi_- ( \gamma )] \bullet \pi_{\gamma', \Gamma'} [ \varphi_- ( \gamma' )] \Big) \Big( i_{\Gamma, \gamma} \big[ \varphi ( \bar\Gamma / \bar\gamma) \big] \bullet  i_{\Gamma', \gamma'} \big[\varphi (\bar\Gamma'/\bar\gamma') \big]\Big) \Big].
\end{eqnarray*}
By using the proposition \ref{concat} we can write:
\begin{eqnarray*}
\varphi_- (\Gamma\Gamma') &=& - P \Big[\varphi (\Gamma)\bullet \varphi (\Gamma') + \varphi_- (\Gamma) \bullet\varphi (\Gamma') +\varphi (\Gamma) \bullet\varphi_- (\Gamma')\\
&+& \sum_{\bar\gamma \subsetneq\; \bar\Gamma \atop \bar\Gamma / \bar\gamma \in \Cal T} \Big(\pi_{\gamma, \Gamma}  [ \varphi_- ( \gamma )] \; i_{\Gamma, \gamma} \big[\varphi \left(\bar\Gamma / \bar\gamma \right)\big]\Big) \bullet \Big(\varphi_- (\Gamma') + \varphi (\Gamma') \Big) \\
&+& \sum_{\bar\gamma' \subsetneq\; \bar\Gamma' \atop \bar\Gamma' / \bar\gamma' \in \Cal T} \Big( \pi_{\gamma', \Gamma'} [ \varphi_- ( \gamma' )] \; i_{\Gamma', \gamma'} \big[\varphi ( \bar\Gamma' / \bar\gamma' )\big]\Big) \bullet\Big( \varphi_- (\Gamma) + \varphi (\Gamma)\Big) \\
&+&\sum_{\bar\gamma \subsetneq\; \bar\Gamma \;,\; \bar\gamma' \subsetneq\; \bar\Gamma' \atop \bar\Gamma / \bar\gamma \;;\;\bar\Gamma' / \bar\gamma' \in \Cal T} \Big(\pi_{\gamma, \Gamma}  [ \varphi_- ( \gamma )] \;i_{\Gamma, \gamma} \big[ \varphi (\bar\Gamma / \bar\gamma) \big]\Big)\bullet\Big( \pi_{\gamma', \Gamma'} [ \varphi_- ( \gamma' )] \;  i_{\Gamma', \gamma'} \big[\varphi (\bar\Gamma' / \bar\gamma') \big]\Big) \Big]\\
&=&\varphi_- (\Gamma)\bullet\varphi_- (\Gamma'),
\end{eqnarray*}
which shows that $\varphi_- $ is a character.
\end{enumerate}
\end{proof}
\subsection {Taylor expansions} 
We adapt here a construction from \cite[\S 9]{EGP} also used by \cite[\S 3.7]{EP}, (see also \cite{KP, KP2}). 
\begin{definition}
Let $\Cal B$ be the commutative algebra defined by \eqref{calB}. For $m \in \mathbb N$ the order $m$ Taylor expansion operator is:
\begin{equation}
P_m \in \mop{End}(\Cal B),  \;\;\;\; P_m f  (v) := \sum_{|\beta| \leq m} \frac{v^\beta}{\beta!}\partial_0^\beta f,\label{pm}
\end{equation}
where $\beta = (\beta_1, ..., \beta_n) \in {\mathbb N}^n$ with the usual notations $\beta \leq \alpha$ iff $\beta_i \leq \alpha_i$  for all $i$, $|\beta|:= \beta_1 + ...+ \beta_n$ as well as 
$$ v^\beta = \prod_{1 \leq k \leq n} v_k^{\beta_k} , \;\;\;\;\;\;\;\;\; \beta! := \prod_{1 \leq k \leq n} \beta_k! ,\;\;\; \;\;\;\;\;\; \partial_0^\beta := \prod_{1 \leq k \leq n} \frac{\partial^{\beta_k}}{\partial v_k^{\beta_k}}_{|v_k =0}.$$ 	
\end{definition}
We can now implement the general momentum scheme using these projections $P_m$. Let $\wt{\Cal H}_{\Cal T} = \bigoplus_n \wt{\Cal H}_{\Cal T, n}$ be the specified Feynman graphs graded bialgebra: we define a Birkhoff decomposition:
\begin{equation}
\varphi = \varphi _-^{\circledast-1} \circledast \varphi_+.
\end{equation}
The components $\varphi_+$ and $\varphi_- $ are given by simple recursive formulas: for any $\bar\Gamma$ of degree zero (i.e without internal edges) we put: $\varphi _-(\bar\Gamma) = \varphi _+(\bar\Gamma) = \varphi (\bar\Gamma)= \un_{V_\Gamma}$. If we assume that $ \varphi_- (\bar\Gamma) $ and $\varphi_+ (\bar\Gamma) $ are known for $ \bar\Gamma $ of degree $ k \leq m-1 $, we have then for any specified graph $\bar\Gamma $ of degree $m$:
\begin{equation}
\varphi_- (\bar\Gamma) = - P_m \Big( \varphi (\Gamma) + \sum_{\bar\gamma \subsetneq\; \bar\Gamma\atop \bar\Gamma / \bar\gamma \in \Cal T} \pi_{\gamma, \Gamma} [\varphi_- (\gamma)] \;i_{\Gamma, \gamma} \Big[ \varphi \left( \Gamma / (\gamma ,\underline{j})\right) \Big ]\Big) \label{b1}
\end{equation}
\begin{equation}
\varphi_+ (\bar\Gamma) = (I- P_m) \Big( \varphi (\Gamma) + \sum_{\bar\gamma \subsetneq\; \bar\Gamma \atop \bar\Gamma / \bar\gamma \in \Cal T} \pi_{\gamma, \Gamma} [\varphi_- (\gamma)] \;i_{\Gamma, \gamma} \Big[ \varphi \left( \Gamma / (\gamma ,\underline{j})\right) \Big ]\Big). \label{bb1}
\end{equation}
The operators $P_m$ form a \textsl{Rota--Baxter family} in the sense of K. Ebrahimi-Fard, J. Gracia-Bondia and F. Patras \cite[Proposition 9.1, Proposition 9.2]{EGP}: the analogue of the Rota-Baxter equality defined by the formula \eqref{Rb} is given by following theorem \cite{EGP, EP}:
\begin{theorem}
Let $\Gamma$ be a graph, and let $f, g \in V_\Gamma$. The Taylor expansion operators fulfil for any $s, t \in \mathbb N$:
\begin{equation}
 ( P_s f ) (P_t g) = P_{s+t} [ (P_s f)g + f (P_t g) - fg]. \label{bb}
\end{equation} 
\end{theorem}
\begin{proof}
Denote by $\mu (f\otimes g)= fg$ the pointwise product on $V_\Gamma$. Using the Leibniz rule: 
\begin{equation}
\partial \circ \mu = \mu \circ (\partial \otimes Id + Id \otimes \partial),
\end{equation}
and the formula 
\begin{equation}\label{fr}
\partial_0^\alpha P_s = \partial_0^\alpha \sum_{|\beta| \leq s} \frac{v \mapsto v^\beta}{\beta!}\partial_0^\beta = \sum_{|\beta| \leq s} \frac{\partial_0^\alpha(v \mapsto v^\beta)}{\beta!}\partial_0^\beta = \left\lbrace
\begin{array}{lcl}
\partial_0^\alpha  \;\;\;\;\; \text{if} \;\; |\alpha| \leq s\\
0  \;\; \;\; \;\; \;\; \text{else},
\end{array}\right. 
\end{equation}
by the formula \eqref{pm} it suffices to check for any multiindex $|\alpha| \leq s+t$ that:
\begin{eqnarray*}
\partial_0^\alpha[ (P_s f)g + f (P_t g) - fg] =  \sum_{\beta \leq \alpha}\binom{\alpha}{\beta} \mu \circ (\partial_0^\beta \otimes \partial_0^{\alpha -\beta})[ (P_s f)\otimes g + f \otimes(P_t g) - f \otimes g]\\
= \sum_{\beta \leq \alpha}\binom{\alpha}{\beta} \big[ (\partial_0^\beta P_s f) (\partial_0^{\alpha -\beta} g) + (\partial_0^\beta f) (\partial_0^{\alpha -\beta} P_t g) - (\partial_0^\beta f) (\partial_0^{\alpha -\beta} g)\big]\\
=\sum_{\beta \leq \alpha}\binom{\alpha}{\beta}  (\partial_0^\beta P_s f) (\partial_0^{\alpha -\beta} P_t g) = \partial_0^\alpha [(P_s f).P_t g)].
\end{eqnarray*}
Here we used that in the middle line, by formula \eqref{fr} the contributions with $|\beta| > s$ or $|\alpha - \beta| > t$ give zero. For example, if $|\alpha - \beta| > t$ then  $|\alpha| - |\beta| > t \Longrightarrow |\beta| < |\alpha| - t$, since $|\alpha| \leq s+t$  then  $|\beta| < s$ such that:
$$ \partial_0^\beta P_s = \partial_0^\beta, \;\;\;\;\text{and}\;\;\;\;  \partial_0^{\alpha -\beta} P_t = 0,$$
then 
$$
\underbrace{(\partial_0^\beta P_s f)}_{\partial_0^\beta f} (\partial_0^{\alpha -\beta} g) + (\partial_0^\beta f) \underbrace{(\partial_0^{\alpha -\beta} P_tg)}_{0} - (\partial_0^\beta f) (\partial_0^{\alpha -\beta} g) = 0,$$
and  
$$(\partial_0^{\beta} P_s f) (\partial_0^{\alpha -\beta} P_t g) = 0.$$
Hence only terms with $|\beta| \leq s$ and $|\alpha - \beta| \leq t$ remain, we obtain:
$$ \partial_0^\beta P_s = \partial_0^\beta ,\;\;\;\; \text{and}\;\;\;\; \partial_0^{\alpha -\beta} P_t = \partial_0^{\alpha -\beta},$$
then we have:
\begin{eqnarray*}
(\partial_0^\beta P_s f)(\partial_0^{\alpha -\beta} g) + (\partial_0^\beta f) (\partial_0^{\alpha -\beta} P_tg) - (\partial_0^\beta f) (\partial_0^{\alpha -\beta} g) &=& (\partial_0^\beta f) (\partial_0^{\alpha -\beta} g)\\
&=& (\partial_0^{\beta} P_s f) (\partial_0^{\alpha -\beta} P_t g).
\end{eqnarray*}
\end{proof}
\begin{theorem}
Let $\wt{\Cal H}_{\Cal T}$ be the specified graphs graded bigebra and $\varphi$ be a character  with values in the unitary commutative algebra $\Cal B$. Further let $P_.: \mathbb N \longrightarrow \mop{End}(\Cal B)$ be an indexed renormalization scheme, that is a family $(P_t)_{t \in \mathbb N}$ of endomorphisms such that:
\begin{equation}\label{fam}
 \mu \circ (P_s \otimes P_t) = P_{s+t} \circ \mu \circ [P_s \otimes Id + Id \otimes P_t - Id \otimes Id],
\end{equation}  
for all $s, t \in \mathbb N$. Then the two maps $\varphi_-$ and $\varphi_+$ defined by \eqref{b1}  and \eqref{bb1}  are two characters.
\end{theorem}
\begin{proof}
We will just prove that $\varphi_-$ is a character. Then  $\varphi_+ = \varphi_- \circledast \varphi$ is also a character. For $\bar\Gamma, \bar\Gamma' \in \mop{ker} \varepsilon$, we write $\varphi_- (\bar\Gamma) =  - P_{|\Gamma|}  ( \bar\varphi (\bar\Gamma) )$ where 
$$\bar\varphi (\bar\Gamma) = \varphi (\Gamma) + \sum_{\bar\gamma \subsetneq\; \bar\Gamma\atop \bar\Gamma / \bar\gamma \in \Cal T} \pi_{\gamma, \Gamma} [\varphi_- (\gamma)] \;i_{\Gamma, \gamma} \big[ \varphi \left( \Gamma / (\gamma ,\underline{j})\right) \big ].$$
For proving this theorem we use the formulas \eqref{b1} and \eqref{fam}.
\begin{eqnarray*}
\varphi_- (\bar\Gamma\bar\Gamma') &=& - P_{|\Gamma\Gamma'|} \Big[\varphi (\Gamma)\bullet \varphi (\Gamma') + \varphi_- (\Gamma) \bullet\varphi (\Gamma') +\varphi (\Gamma) \bullet\varphi_- (\Gamma')\\
&+& \sum_{\bar\gamma \subsetneq\; \bar\Gamma \atop \bar\Gamma / \bar\gamma \in \Cal T} \Big(\pi_{\gamma, \Gamma}  [ \varphi_- ( \gamma )] \; i_{\Gamma, \gamma} \big[\varphi \left(\bar\Gamma / \bar\gamma \right)\big]\Big) \bullet \Big(\varphi_- (\Gamma') + \varphi (\Gamma') \Big) \\
&+& \sum_{\bar\gamma' \subsetneq\; \bar\Gamma' \atop \bar\Gamma' / \bar\gamma' \in \Cal T} \Big( \pi_{\gamma', \Gamma'} [ \varphi_- ( \gamma' )] \; i_{\Gamma', \gamma'} \big[\varphi ( \bar\Gamma' / \bar\gamma' )\big]\Big) \bullet\Big( \varphi_- (\Gamma) + \varphi (\Gamma)\Big) \\
&+&\sum_{\bar\gamma \subsetneq\; \bar\Gamma \;,\; \bar\gamma' \subsetneq\; \bar\Gamma' \atop \bar\Gamma / \bar\gamma \;;\;\bar\Gamma' / \bar\gamma' \in \Cal T} \Big(\pi_{\gamma, \Gamma}  [ \varphi_- ( \gamma )] \;i_{\Gamma, \gamma} \big[ \varphi (\bar\Gamma / \bar\gamma) \big]\Big)\bullet\Big( \pi_{\gamma', \Gamma'} [ \varphi_- ( \gamma' )] \;  i_{\Gamma', \gamma'} \big[\varphi (\bar\Gamma' / \bar\gamma') \big]\Big) \Big]\\
&=&- P_{|\Gamma| + |\Gamma'|} \Big[ \Big(\varphi (\Gamma) +\sum_{\bar\gamma \subsetneq\; \bar\Gamma \atop \bar\Gamma / \bar\gamma \in \Cal T} \big(\pi_{\gamma, \Gamma}  [ \varphi_- ( \gamma )] \; i_{\Gamma, \gamma} \big[\varphi \left(\bar\Gamma / \bar\gamma \right)\big]\big) \Big)\\
&&\bullet \Big(\varphi (\Gamma') + \sum_{\bar\gamma' \subsetneq\; \bar\Gamma' \atop \bar\Gamma' / \bar\gamma' \in \Cal T} \big( \pi_{\gamma', \Gamma'} [ \varphi_- ( \gamma' )] \; i_{\Gamma', \gamma'} \big[\varphi ( \bar\Gamma' / \bar\gamma' )\big]\big) \Big)\\
&+&\varphi_- (\Gamma) \bullet \Big(\varphi (\Gamma') + \sum_{\bar\gamma' \subsetneq\; \bar\Gamma' \atop \bar\Gamma' / \bar\gamma' \in \Cal T} \big( \pi_{\gamma', \Gamma'} [ \varphi_- ( \gamma' )] \; i_{\Gamma', \gamma'} \big[\varphi ( \bar\Gamma' / \bar\gamma' )\big]\big)\Big)\\
&+& \varphi_- (\Gamma') \bullet \Big(\varphi (\Gamma) + \sum_{\bar\gamma \subsetneq\; \bar\Gamma \atop \bar\Gamma / \bar\gamma \in \Cal T} \big( \pi_{\gamma, \Gamma} [ \varphi_- (\gamma)] \; i_{\Gamma, \gamma} \big[\varphi ( \bar\Gamma / \bar\gamma)\big]\big) \Big)\Big]\\
&=&- P_{|\Gamma| + |\Gamma'|} \Big[\bar\varphi (\bar\Gamma)\bullet \bar\varphi (\bar\Gamma') - P_{|\Gamma|}(\bar\varphi (\bar\Gamma))\bullet\bar\varphi (\bar\Gamma') - P_{|\Gamma'|}(\bar\varphi (\bar\Gamma'))\bullet \bar\varphi (\bar\Gamma) \Big]\\
&=& P_{|\Gamma| + |\Gamma'|} \Big[ P_{|\Gamma'|}(\bar\varphi (\bar\Gamma'))\bullet \bar\varphi (\bar\Gamma) + P_{|\Gamma|}(\bar\varphi (\bar\Gamma))\bullet\bar\varphi (\bar\Gamma') - \bar\varphi (\bar\Gamma)\bullet \bar\varphi (\bar\Gamma')\Big]\\
&=& \big( P_{|\Gamma|}(\bar\varphi (\bar\Gamma))\big)\bullet \big( P_{|\Gamma'|}(\bar\varphi (\bar\Gamma'))\big)\\
&=& \varphi_- (\bar\Gamma)\bullet \varphi_- (\bar\Gamma').
\end{eqnarray*}
\end{proof}



\begin{thebibliography}{abcdsfgh}
{\small{
\bibitem{SA} S. Agarwala, \textit{The geometry of renormalisation}, PhD thesis, Johns Hopkins University (2009).
\bibitem{A.D2000}A. Connes, D. Kreimer, \textit{Renormalization in quantum field theory and the Riemann-Hilbert problem. I. The Hopf algebra structure of graphs and the main theorem}, Comm. Math. Phys. 210, $n^{\circ} 1$, 249-273 (2000).
\bibitem{ad01}A. Connes, D. Kreimer, \textit{Renormalization in quantum field theory and the Riemann-Hilbert problem. II. The $\beta$-function,
diffeomorphisms and the renormalization group}, Comm. in Math. Phys. 216, 215-241 (2001).
\bibitem{ad98}A. Connes, D. Kreimer, \textit{Hopf algebras, renormalization and noncommutative geometry}, Comm. in Math. Phys. 199,203-242 (1998).
\bibitem{CM}A. Connes, M. Marcolli, \textsl{Noncommutative geometry, quantum fields and motives}, preprint, http://www.alainconnes.org/fr/bibliography.php.
\bibitem{k.g.m} K. Ebrahimi-Fard, L. Guo, D. Manchon, \textit{Birkhoff type decompositions and the Baker-Campbell-Hausdorff recursion}, Comm. Math. Phys. 267, 821-845 (2006).
\bibitem{EGP}K. Ebrahimi-Fard, J. Gracia-Bondia, F. Patras, 
\textit{A Lie theoretic approach to renormalization}, Comm. Math. Phys. 276, 519-549 (2007).
\bibitem{KP} K. Ebrahimi-Fard, F. Patras, \textit{Exponential renormalization}, Ann. Henri Poincar\'e 11, 943-971 (2010).
\bibitem{KP2} K. Ebrahimi-Fard, F. Patras, \textit{Exponential renormalization II. Bogoliubov's R-operation and momentum
subtraction schemes}, J. Math. Phys. 53, 083505 (2012).
\bibitem{dk98}D. Kreimer, \textit{On the Hopf algebra structure of perturbative quantum field theories}, Adv. Theor. Math. Phys. 2, 303-334
(1998).
\bibitem{Dm11}D. Manchon, \textit{On bialgebra and Hopf algebra of oriented graphs}, Confluentes Math. Volume 04, No. 1 (2012). 
\bibitem{Dm08}D. Manchon, \textit{Hopf algebras and renormalisation}, Handbook of Algebra,Vol. 5 (M. Hazewinkel ed.), 365-427 (2008).
\bibitem{Sw69}M.~E.~Sweedler, \textit{Hopf algebras}, Benjamin, New-York (1969).
\bibitem{EP} E. Panzer, \textit{Hopf algebraic renormalization of Kreimer's toy model}. Arxiv: math.QA: 1202.3552v1 (2012).
\bibitem{wvs}W.D. van Suijlekom, \textit{The structure of renormalization Hopf algebras for gauge theories I: representing Feynman
graphs on BV-algebras}, Commun. Math. Phys. 290  291-319 (2009).
\bibitem{wvs06}W.D. van Suijlekom, \textit{The Hopf algebra of Feynman graphs in QED}, letters in Math. Phys. 77, 265-281 (2006). 
}}
\end{thebibliography}
\end{document}